\documentclass[a4paper,12pt]{amsart}
\usepackage{amsmath}
\usepackage{amsthm}
\usepackage{amssymb}
\usepackage{amscd}
\usepackage[all]{xy}

\newtheorem{thm}{Theorem}%[section]

\newtheorem{prop}[thm]{Proposition}%[section]
\newtheorem{lemma}{Lemma}[section]
\newtheorem{ex}{Example}%[section]
\newtheorem{rmk}[]{Remark}[section]

\theoremstyle{remark}
\newtheorem{question}{Question}%[section]

\newcommand{\proj}{\mathop{\rm Proj}\nolimits}
\newcommand{\im}{\mathop{\rm Im}\nolimits}

\newcommand{\calhom}{\mathop{{\mathcal Hom}}\nolimits}
\DeclareMathOperator{\Hom}{Hom}
\DeclareMathOperator{\Tor}{Tor}
\DeclareMathOperator{\Ext}{Ext}
\newcommand{\calext}{\mathop{{\mathcal Ext}}\nolimits}
\DeclareMathOperator{\RHom}{RHom}

\DeclareMathOperator{\lotimes}{\otimes^{L}}
\newcommand{\caltor}{\mathop{{\mathcal T\!or}}\nolimits}
\DeclareMathOperator{\Ker}{Ker}
\DeclareMathOperator{\Coker}{Coker}

\newcommand{\coker}{\mathop{\rm coker}\nolimits}

\DeclareMathOperator{\End}{End}
\DeclareMathOperator{\length}{length}

\DeclareMathOperator{\module}{mod}

\newcommand{\tors}{\mathop{\rm tors}\nolimits}

\makeatletter

\@addtoreset{equation}{section}
\makeatother

\title[
Nef vector bundles on a projective space
with $c_1=3$
]{
Nef vector bundles on a projective space
with first Chern class three
}

\thanks{
This work was partially supported by 
JSPS KAKENHI (C) Grant Number 15K04810.
}

\author{Masahiro Ohno
}

\address{Graduate School of Informatics and Engineering,
The University of Electro-Communications,
Chofu-shi,
Tokyo, 182-8585 Japan
}

\email{masahiro-ohno@uec.ac.jp}



\begin{document}
\begin{abstract}
We classify nef vector bundles on a projective space
with first Chern class three
over an algebraically closed field of characteristic zero;
we see, in particular, that
these nef vector bundles are globally generated
if the second Chern class is less than eight,
and that
there exist
nef but non-globally generated vector bundles
with 
second Chern class 
eight and nine 
on a projective plane.
\end{abstract}

\maketitle

\section{Introduction}\label{Introduction}
Let $\mathcal{E}$ 
be a nef vector bundle of rank $r$ on a projective space $\mathbb{P}^n$
over an algebraically closed field $K$ of characteristic zero.
Let $c_1$ be the first Chern class of $\mathcal{E}$.
Then $c_1$ is non-negative
since $\mathcal{E}$ is nef.
In \cite[Theorem 1]{pswnef},
Peternell-Szurek-Wi\'{s}niewski
classified such $\mathcal{E}$'s in case $c_1\leq 2$,
based on 
the study~\cite{swNagoya}
of Szurek-Wi\'{s}niewski.
If $c_1\leq 2$ and $n\geq 2$, then $\mathbb{P}(\mathcal{E})$ is a Fano manifold,
and their proof is based on analysis of contraction morphisms of extremal rays.
In \cite[\S 6]{resolution}, a different proof of the classification in case $c_1\leq 2$
is given, based on analysis of some twist $\mathcal{E}(d)$ of $\mathcal{E}$ with 
the 
full strong exceptional sequence 
$\mathcal{O},\mathcal{O}(1),\dots,\mathcal{O}(n)$ 
of line bundles.

In this paper we continue our approach to 
classify
nef vector bundles in the next case $c_1=3$.
Note here that if $c_1=3$ then the anti-canonical bundle of $\mathbb{P}(\mathcal{E})$ is nef if $n\geq 2$
and ample if $n\geq 3$.
Moreover if $c_1=3$ then $0\leq c_2\leq c_1^2=9$,
where $c_2$ denotes the second Chern class of $\mathcal{E}$.
The main result of this paper is as follows.

\begin{thm}\label{c_1=3c_2<8}
Let $\mathcal{E}$ be a nef vector bundle of rank $r$
on a projective space $\mathbb{P}^n$
with $c_1=3$.
Then $c_2$ and 
$\mathcal{E}$
satisfy
one of the following:
\begin{enumerate}
\item[$(1)$] $c_2=0$ and $\mathcal{E}\cong\mathcal{O}(3)\oplus \mathcal{O}^{\oplus r-1}$;
\item[$(2)$] $c_2=2$ and $\mathcal{E}\cong\mathcal{O}(2)\oplus\mathcal{O}(1)\oplus \mathcal{O}^{\oplus r-2}$;
\item[$(3)$] $c_2=3$ and $\mathcal{E}\cong \mathcal{O}(1)^{\oplus 3}\oplus \mathcal{O}^{\oplus r-3}$;
\item[$(4)$] $c_2=3$, $n=2$, and $\mathcal{E}\cong T_{\mathbb{P}^2}\oplus \mathcal{O}^{\oplus r-2}$;
\end{enumerate}
$($In the following, $\mathcal{E}$ fits in one of the following exact sequences.$)$
\begin{enumerate}
\item[$(5)$] $c_2=3$ and 
$0\to\mathcal{O}(-1)\to\mathcal{O}(2)\oplus \mathcal{O}^{\oplus r}\to \mathcal{E}\to 0$;
\item[$(6)$] $c_2=4$ and 
$0\to\mathcal{O}(-1)\to\mathcal{O}(1)^{\oplus 2}\oplus\mathcal{O}^{\oplus r-1}\to \mathcal{E}\to 0$;
\item[$(7)$] $c_2=4$, $n=3$, and 
$0\to\mathcal{O}(-2)\to\mathcal{O}(-1)^{\oplus 4}\to\mathcal{O}(1)\oplus\mathcal{O}^{\oplus r+2}\to \mathcal{E}\to 0$;
\item[$(8)$] $c_2=4$, $n=4$, and 
\[0\to\mathcal{O}(-3)\to \mathcal{O}(-2)^{\oplus 5}\to\mathcal{O}(-1)^{\oplus 10}\to\mathcal{O}^{\oplus r+6}\to \mathcal{E}\to 0;\]
\item[$(9)$] $c_2=5$ and 
$0\to\mathcal{O}(-1)^{\oplus 2}\to\mathcal{O}(1)\oplus\mathcal{O}^{\oplus r+1}\to \mathcal{E}\to 0$;
\item[$(10)$] $c_2=5$, $n=3$ or $4$,
and 
$0\to\mathcal{O}(-2)\to \mathcal{O}(-1)^{\oplus 5}\to\mathcal{O}^{\oplus r+4}\to \mathcal{E}\to 0$;
\item[$(11)$] $c_2=6$ and 
$0\to\mathcal{O}(-1)^{\oplus 3}\to\mathcal{O}^{\oplus r+3}\to \mathcal{E}\to 0$;
\item[$(12)$] $c_2=6$ and 
$0\to\mathcal{O}(-2)\to\mathcal{O}(1)\oplus\mathcal{O}^{\oplus r}\to \mathcal{E}\to 0$;
\item[$(13)$] $c_2=7$ and 
$0\to\mathcal{O}(-2)\oplus\mathcal{O}(-1)\to\mathcal{O}^{\oplus r+2}\to \mathcal{E}\to 0$;
\item[$(14)$] $c_2=8$, $n=2$, and 
$
0\to \mathcal{O}(-2)^{\oplus 2}\to \mathcal{O}^{\oplus r+1}\oplus \mathcal{O}(-1)\to \mathcal{E}\to 0$;
\item[$(15)$] $c_2=9$ and
$0\to \mathcal{O}(-3)\to \mathcal{O}^{\oplus r+1}\to \mathcal{E}\to 0$;
\item[$(16)$] $c_2=9$, $n=2$, and 
$0\to \mathcal{O}(-2)^{\oplus 3}\to \mathcal{O}^{\oplus r}\oplus \mathcal{O}(-1)^{\oplus 3}\to \mathcal{E}\to 0$;
\item[$(17)$] $c_2=9$, $n=2$, and 
\[0\to \mathcal{O}(-3)^{\oplus r}\to \mathcal{O}(-2)^{\oplus 3r+3}
\to \mathcal{O}(-1)^{\oplus 3r+3}\to \mathcal{E}\to 0;\]
\item[$(18)$] $c_2=9$, $n=2$, and 
$0\to \mathcal{O}(-2)^{\oplus 4}\to \mathcal{O}(1)\oplus\mathcal{O}^{\oplus r-3}\oplus \mathcal{O}(-1)^{\oplus 6}\to \mathcal{E}\to 0$;
\item[$(19)$] $c_2=9$, $n\geq 3$, $c_3=27$, and $h^0(\mathcal{E}(-1))=1$;
\item[$(20)$] $c_2=9$, $n\geq 4$, $c_3=27$, $h^0(\mathcal{E}(-1))=0$, and $h^{n-3}(\mathcal{E}(2-n))=1$.
\end{enumerate}
\end{thm}

Note that every case except the cases (18), (19), and (20) in Theorem~\ref{c_1=3c_2<8} is effective:
examples of $\mathcal{E}$ in case (7) are 
$\mathcal{O}(1)\oplus \mathcal{O}^{\oplus r-4}\oplus \Omega_{\mathbb{P}^3}(2)$
and $\mathcal{O}(1)\oplus \mathcal{O}^{\oplus r-3}\oplus \mathcal{N}(1)$ where $\mathcal{N}$ is a null correlation bundle on $\mathbb{P}^3$
(see Remark~\ref{Rmk for the case (7)});
in case (8), $\mathcal{E}$ is given by a locally free resolution 
in terms of $\mathcal{O}(-3)$-twist of the full strong exceptional sequence
$\mathcal{O}, \mathcal{O}(1),\dots, \mathcal{O}(4)$ in accordance with \cite{resolution},
but $\mathcal{E}$ is in fact isomorphic to $\Omega_{\mathbb{P}^4}(2)\oplus\mathcal{O}^{\oplus r-4}$
(see Remark~\ref{Rmk for the case (8)});
in case (10), if $n=4$, $\mathcal{E}$ is nothing but an extension of 
the Tango 
bundle
\cite{MR0401766}
(see also \cite[Chap.\ I \S 4.3]{oss})
on $\mathbb{P}^4$ by a trivial bundle $\mathcal{O}^{\oplus r-3}$,
so that $\Omega_{\mathbb{P}^4}^2(3)$ is a typical example
(see Remark~\ref{TrautmannVetter}),
and if $n=3$, the restriction of such bundle to a hyperplane $\mathbb{P}^3$ is an example;
the case (14) comes from Proposition~\ref{exampleofnefbutnonGGvb} below 
(see \cite{Nefofc1=3c2=8OnPN} for details);
an example in case (16) is given in Example~\ref{Example of the case (16)};
the exact sequence in case (17) in fact derives from 
Example~\ref{ExampleOfNonGGWithc_2=9}.
On the other hand,
it is uncertain whether nef vector bundles exist or not in case (18), (19), or (20).
Note that
neither in case (19) nor in case(20) exist nef vector bundles if they does not exist  in case (18),
since the restriction to a plane of a nef vector bundle in case (19) or (20) lies in the case (18).

Anghel-Manolache \cite{MR3119690} and Sierra-Ugaglia \cite{MR3120618} classified 
globally generated vector bundles on a projective space
with first Chern class three.
Since global generation implies nefness, Theorem~\ref{c_1=3c_2<8} is a generalization of their results.
We also note that 
Langer
\cite{MR1633159}
classified smooth Fano $4$-folds with adjunction theoretic scroll structure
and $b_2=2$; 
his classification includes
that of nef and big rank $2$ bundles with $c_1=3$ on $\mathbb{P}^2$ and $\mathbb{P}^3$.

As in the proof in 
case 
$c_1\leq 2$ in \cite{resolution},
the main feature of our proof of Theorem~\ref{c_1=3c_2<8}
is  an application of the spectral sequence deduced 
in \cite[Theorem 1]{MR3275418}
from Bondal's theorem~\cite[Theorem 6.2]{MR992977}.
Besides this spectral sequence, some of the key ingredients of our proof are
the Riemann-Roch formula, 
the Kawamata-Viehweg vanishing theorem~\cite{MR0675204} \cite{MR0667459},
and the non-negativity of Chern classes of nef vector bundles
(see, e.g., \cite[Theorem 8.2.1]{MR2095472}).

Whereas global generation implies nefness, the converse is not true in general;
indeed, the (scheme-theoretic) support 
of the cokernel of the evaluation map $H^0(\mathcal{E})\otimes \mathcal{O}\to \mathcal{E}$
in cases (14), (16), and  (17)
in Theorem~\ref{c_1=3c_2<8}
is,  respectively, 
a point $w$, 
a cubic curve $E$,
and 
the whole $\mathbb{P}^2$.
(If there exists an example in case (18), 
its corresponding support is also the whole $\mathbb{P}^2$.
Hence examples in cases (18), (19), or (20) shall be also nef but non-globally generated
if they exist.)
In fact,
the evaluation map in case (14) 
fits in the exact sequence in 
the following proposition.

\begin{prop}\label{exampleofnefbutnonGGvb}
Given an integer $r\geq 2$ and a point $w$ in a projective plane $\mathbb{P}^2$, 
there exists a vector bundle $\mathcal{E}$ fitting in an exact sequence
\[
0\to \mathcal{O}(-3)\to \mathcal{O}^{\oplus r+1}\to \mathcal{E}\to k(w)\to 0
\]
where $k(w)$ denotes the residue field of the point $w$.
Moreover a vector bundle $\mathcal{E}$ fitting in the sequence above is nef
but non-globally generated with 
$c_1=3$
and 
$c_2=8$.
\end{prop}

In \cite[Theorem 1.1]{Nefofc1=3c2=8OnPN}, 
we classified
nef vector bundles on a projective space 
with $c_1=3$ and $c_2=8$
(based on the previous version of this manuscript): we showed that such bundles exist only on a projective plane
and every such bundle derives from the exact sequence given in Proposition~\ref{exampleofnefbutnonGGvb}
and fits in the exact sequence (14) in Theorem~\ref{c_1=3c_2<8}.
Note that the parts of the previous version of the manuscript on which the argument in \cite{Nefofc1=3c2=8OnPN} are relied are kept 
the same in this major revised manuscript.

This paper is organized as follows.
In \S \ref{Preliminaries}, we recall Bondal's theorem~\cite[Theorem 6.2]{MR992977}
and its related results including the spectral sequence deduced in \cite[Theorem 1]{MR3275418}.
The results in \S \ref{Preliminaries} are fundamental throughout this paper.
In \S \ref{SetUp}, we begin our proof(s) of Theorem~\ref{c_1=3c_2<8}
(and of Theorem~\ref{globalgeneration} below);
based on results
in \cite{resolution}, 
we reduce the problem to the case where 
$H^0(\mathcal{E}(-2))=0$. We also show in \S \ref{SetUp} that 
this reduction enables us to assume that 
$3\leq c_2$. 
Several other formulas---such as the Riemann-Roch formulas---used repeatedly in this paper
are also presented in \S \ref{SetUp}.
In \S \ref{Lemmas}, we collect some lemmas 
needed later.
In \S \ref{KeyLemma}, we give a key lemma, Lemma~\ref{key},
which together with exact sequence (\ref{E_2^{-1,1}quotient}) in \S \ref{Set-up for the two-dimensional case} 
is crucial for the whole proof.
In \S \ref{shortcut},
we show that nef vector bundles on a projective space with $c_1=3$ and $c_2\leq 7$ are globally generated
(Theorem~\ref{globalgeneration}).
In \S \ref{Proof of Main Theorem}, we give a proof of Theorem~\ref{c_1=3c_2<8};
we omit the proof 
in case $c_2\leq 7$,
since we have Theorem~\ref{globalgeneration}
and globally generated vector bundles with $c_1=3$ are classified in \cite{MR3119690} and \cite{MR3120618}.
Similarly the proof in case $c_2=8$ is omitted
since it is already given in \cite{Nefofc1=3c2=8OnPN};
thus we only give a proof 
in case $c_2=9$.
Note that our presentation of the classification in case $c_2\leq 7$
is different from both of \cite{MR3119690} and \cite{MR3120618};
the situation is the same as in \cite{MR3119690} and \cite{MR3120618}:
the presentations of the classification in \cite{MR3119690} and \cite{MR3120618}
are different because their methods of proofs are different;
similarly our presentation differs from those of \cite{MR3119690} and \cite{MR3120618}
because of our new proof.
A reader who wonders where the presentation of Theorem~\ref{c_1=3c_2<8} 
in 
case $c_2\leq 7$
comes from may find its proof in the previous version (arXiv 1604.05847 version 4)
of the manuscript.
In \S \ref{RmksOnMainTheorem}
we 
give several remarks 
related to the cases (7), (8), and (10) 
and an example of the case (16) 
in Theorem~\ref{c_1=3c_2<8}.
We end \S \ref{RmksOnMainTheorem}
with a question
about the cokernel of the evaluation map $H^0(\mathcal{E})\otimes \mathcal{O}\to \mathcal{E}$
of a nef vector bundle in case (16) of Theorem~\ref{c_1=3c_2<8}.
In \S \ref{nefbutNonGG}, we give a proof of Proposition~\ref{exampleofnefbutnonGGvb}.
Finally note that,
since globally generated vector bundles are nef, some properties of globally generated vector bundles
also hold more generally for nef vector bundles, but some do not;
in \S \ref{PropertyOfNefNonGG}, 
we give two examples 
which illustrate that a typical exact sequence
for globally generated vector 
bundles---related to the degeneracy locus of general global sections---does 
not necessarily exist
for nef vector bundles.

\subsection{Notation and conventions}\label{convention}
Throughout this paper
we work over an algebraically closed field $K$
of characteristic zero.
Basically we follow the standard notation and terminology in algebraic
geometry. 
For 
a vector bundle $\mathcal{E}$,
$\mathbb{P}(\mathcal{E})$ denotes $\proj S(\mathcal{E})$,
where $S(\mathcal{E})$ denotes the symmetric algebra of $\mathcal{E}$.
The tautological line bundle $\mathcal{O}_{\mathbb{P}(\mathcal{E})}(1)$
is also denoted by $H(\mathcal{E})$.
We denote by $\mathcal{E}^{\vee}$ the dual of $\mathcal{E}$.
For a 
coherent sheaf $\mathcal{F}$ on a smooth projective variety $X$,
we denote by $c_i(\mathcal{F})$ the $i$-th Chern class of $\mathcal{F}$.
In particular, 
$c_i$ stands for $c_i(\mathcal{E})$
of the nef vector bundle $\mathcal{E}$ we are dealing with.
We say that a vector bundle is 
(non-)globally generated
if it is (not) generated by global sections.
For coherent sheaves $\mathcal{F}$ and $\mathcal{G}$ on $X$,
$h^q(\mathcal{F})$ denotes $\dim H^q(\mathcal{F})$,
and $\hom(\mathcal{F},\mathcal{G})$ denotes
$\dim \Hom(\mathcal{F},\mathcal{G})$.
For any closed subscheme $Z$ in $\mathbb{P}^n$, denote by $\mathcal{I}_Z$
the ideal sheaf of $Z$ in $\mathbb{P}^n$.
%Moreover the inverse image ideal sheaf of $\mathcal{I}_Z$ to a curve $C\subset \mathbb{P}^n$
%is denoted by $\mathcal{I}_Z\cdot \mathcal{O}_C$,
%and the $d$-th twist of it is denoted by $\mathcal{I}_Z\cdot \mathcal{O}_C(d)$.
Finally we refer to \cite{MR2095472} for the definition
and basic properties of nef vector bundles.

\section{Preliminaries}\label{Preliminaries}
In our proof of Theorem~\ref{c_1=3c_2<8},  we shall apply repeatedly a spectral sequence deduced in \cite[Theorem 1]{MR3275418}.
So we shall recall this sequence in this section.
Note that this sequence is 
a corollary of Bondal's theorem~\cite[Theorem 6.2]{MR992977}
as can be seen below.

Let $X$ be a smooth projective variety over $K$,
$D^b(X)$ the bounded derived category of 
the abelian category of coherent sheaves on $X$.
Assume that there exists a full strong exceptional sequence 
$G_0,\dots,G_m$ in $D^b(X)$,
and let 
$G$ be the direct sum $\oplus_{j=0}^mG_j$ 
in $D^b(X)$.

Recall that if $X=\mathbb{P}^n$
then $\mathcal{O},\mathcal{O}(1),\dots,\mathcal{O}(n)$
is a full strong exceptional sequence in $D^b(\mathbb{P}^n)$
by Beilinson's theorem~\cite[Theorem]{MR0509388}.

Let $A$ be the endomorphism ring $\End_{D^b(X)}(G)$ of $G$,
and let $e_j$ be the composite of the projection 
$G\to G_j$ and the inclusion $G_j\to G$.
Then $e_j\in A$. Define a right $A$-module $P_j$ by $P_j=e_jA$. 
The natural isomorphism $A\cong \oplus_{j=0}^mP_j$ of right $A$-modules implies that $P_j$ is a projective right $A$-module.
Let $S_j$ be the simple right $A$-module
such that $S_je_j\cong K$ and $S_je_k\cong 0$ for all $k\neq j$
and $0\leq k\leq m$.

Let $D^b(\module A)$ be the bounded derived category of 
the abelian category $\module A$ of finitely generated right $A$-modules.
Bondal's theorem~\cite[Theorem 6.2]{MR992977}
states that $\RHom (G,\bullet):D^b(X)\to D^b(\module A)$ is an exact equivalence.
Since $\bullet\lotimes G:D^b(\module A)\to D^b(X)$ is a quasi-inverse of $\RHom (G,\bullet)$,
we have a natural isomorphism 
\[\RHom(G,\bullet)\lotimes G\cong \bullet\]
of functors on $D^b(X)$.
If we write down this isomorphism for a coherent sheaf $F$ on $X$ in terms of a spectral sequence,
we obtain the following spectral sequence (\cite[Theorem~1]{MR3275418})
\begin{equation}\label{BondalSpectralSequence}
E_2^{p,q}=\caltor_{-p}^A(\Ext^q(G,F),G)
\Rightarrow
E^{p+q}=
\begin{cases}
F& \textrm{if}\quad  p+q= 0\\
0& \textrm{if}\quad  p+q\neq 0.
\end{cases}
\end{equation}
We call this sequence the Bondal spectral sequence.
In practice, in order to apply the Bondal spectral sequence (\ref{BondalSpectralSequence}), we need to compute 
$E_2^{p,q}$.
One way to compute 
$E_2^{p,q}$
is by definition: $E_2^{p,q}=\caltor_{-p}^A(\Ext^q(G,F),G)$,
i.e., through a projective resolution of the right 
$A$-module $\Ext^q(G,F)$.
Recall here (see \cite[Lemma 2.1]{resolution} for a proof)
that a finitely generated right $A$-module $\Ext^q(G,F)$ has a projective resolution 
of the form
\begin{equation}\label{projresolingeneral}
0\to P_0^{\oplus e_{m,0}}\to\dots\to \bigoplus_{j=0}^{m-l}P_{j}^{\oplus e_{l,j}}\to\dots\to
\bigoplus_{j=0}^{m}P_{j}^{\oplus e_{0,j}}\to \Ext^q(G,F)\to 0
\end{equation}
where $e_{0,j}
=\dim \Ext^q(G_j,F)$ for 
$0\leq j\leq m$
and $e_{l,j}$ is determined inductively for 
$l\geq 1$ and $j\leq m-l$  by the following formula:
\[e_{l,j}=\sum_{j<k}e_{l-1,k}
\hom
(G_j,G_k).\]
We shall freely use the following isomorphism:
\[P_j\lotimes_AG=P_j\otimes_AG\cong G_j.\]
This isomorphism together with 
(\ref{projresolingeneral}) implies that 
$E_2^{p,q}$
is the $(-p)$-th homology of the following complex
\[
0\to G_0^{\oplus e_{m,0}}\to\dots\to \bigoplus_{j=0}^{m-l}G_{j}^{\oplus e_{l,j}}\to\dots\to
\bigoplus_{j=0}^{m}G_{j}^{\oplus e_{0,j}}
\to 0.
\]

Throughout this paper,
we set
$X=\mathbb{P}^n$, $m=n$,
and $G_j=\mathcal{O}(j)$ for $0\leq j\leq n$,
and we fix the notation $G_j$, 
$G$, $A$, $P_j$, and $S_j$
as above.

A typical projective resolution 
of 
the form (\ref{projresolingeneral})
in this paper
is in the case where $q=0$ and $F=\mathcal{E}(d)$ for some non-negative integer $d$.
For example, if $\hom (\mathcal{O}(2),\mathcal{E}(d))=0$, we frequently 
and sometimes implicitly consider a projective resolution 
of the form
\begin{equation}\label{typicalprojresol}
0\to P_0^{\oplus (n+1)e_{0,1}}\to P_1^{\oplus e_{0,1}}\oplus P_0^{\oplus e_{0,0}}\to \Hom (G,\mathcal{E}(d))\to 0
\end{equation}
where $e_{0,0}=\hom (\mathcal{O},\mathcal{E}(d))$ and $e_{0,1}=\hom (\mathcal{O}(1),\mathcal{E}(d))$.

Finally note that the Bott formula 
\cite[p.\ 8]{oss}
implies 
$\RHom(G,\Omega_{\mathbb{P}^n}^j(j))\cong S_j[-j]$ for 
$0\leq j\leq n$.
Hence we have isomorphisms  
\begin{equation}\label{Sjpullback}
S_j\lotimes_AG\cong \Omega_{\mathbb{P}^n}^j(j)[j].
\end{equation}
for $0\leq j\leq n$.
In particular,
\begin{equation}\label{Snpullback}
S_n\lotimes_AG\cong \mathcal{O}(-1)[n].
\end{equation}
Note that this gives another way to compute $E_2^{p,q}=\mathcal{H}^{p}(\Ext^q(G,\mathcal{E}(d))\lotimes_AG)$:
through a filtration of $\Ext^q(G,\mathcal{E}(d))$ with subquotients the direct sums of the simple modules $S_j$.
For example, if $\Ext^q(G,\mathcal{E}(d))\cong S_j$ then 
$E_2^{p,q}
=\mathcal{H}^{p}(\Omega_{\mathbb{P}^n}^j(j)[j])$,
and thus $E_2^{p,q}=0$ if $p\neq -j$
and $E_2^{-j,q}=\Omega_{\mathbb{P}^n}^j(j)$.
In particular, if $\Ext^q(G,\mathcal{E}(d))\cong S_n$ then $E_2^{-n,q}=\mathcal{O}(-1)$ and $E_2^{p,q}=0$ if $p\neq -n$.
We shall also use these formulas frequently.

\section{Set-up and 
formulas for the proofs of Theorems~\ref{c_1=3c_2<8}
and \ref{globalgeneration}}\label{SetUp}

In this section,
we give some preparatory parts of our proofs
of Theorems~\ref{c_1=3c_2<8}
and \ref{globalgeneration},
and we also collect several formulas needed later.

Let $\mathcal{E}$ be a nef vector bundle of rank $r$ on a projective space $\mathbb{P}^n$
with $c_1=3$.
If 
$\Hom (\mathcal{O}(3),\mathcal{E})\neq 0$,
then 
$\mathcal{E}\cong \mathcal{O}(3)\oplus \mathcal{O}^{\oplus r-1}$
by \cite[Proposition 5.2 and Remark 5.3]{resolution}. 
This is the case (1) of Theorem~\ref{c_1=3c_2<8}.
Assume that $\Hom (\mathcal{O}(3),\mathcal{E})= 0$.
Then $r\geq 2$.
If 
$\Hom (\mathcal{O}(2),\mathcal{E})\neq 0$,
then 
it follows from \cite[Theorem 6.4]{resolution} that 
$\mathcal{E}$ is in the case (2) or (5) of Theorem~\ref{c_1=3c_2<8}.

In the rest of our proof, we always assume that 
\begin{equation}\label{first assumption}
H^0(\mathcal{E}(-2))=\Hom (\mathcal{O}(2),\mathcal{E})= 0.
\end{equation}

Recall that the Kodaira vanishing theorem implies 
that 
\begin{equation}\label{first vanishing}
H^q(\mathcal{E}|_{L^l}(3-k))=0
\end{equation}
for all $q>0$,
all $l$-dimensional linear subspaces $L^l\subseteq \mathbb{P}^n$,
and all $k\leq l$,
since $\mathcal{E}$ is a nef vector bundle with $c_1=3$
(see \cite[Lemma 4.1 (1)]{resolution} for a proof).
If
$H(\mathcal{E}|_{L^l})$ is 
big
in addition,
then the Kawamata-Viehweg vanishing theorem implies that
\begin{equation}\label{KVvanishing in arbitrary dim}
H^q(\mathcal{E}|_{L^l}(2-k))=0
\end{equation}
for all $q>0$ and all $k\leq l$(see \cite[Lemma 4.1 (2)]{resolution} for a proof).

Since $\mathcal{E}$ is nef,
$H^q(\mathcal{E}|_{L}(1))=0$ for all $q>0$ and all lines $L$ in $\mathbb{P}^n$.
Together with (\ref{first vanishing}), this implies that 
$
H^2(\mathcal{E}|_{L^2})=0
$
for any plane $L^2\subseteq\mathbb{P}^n$.
This vanishing $H^2(\mathcal{E}|_{L^2})=0$ then implies that 
$
H^2(\mathcal{E}|_{L^2}(-1))=0
$
since $H^q(\mathcal{E}|_{L})=0$ for all $q>0$. 
Moreover we have 
$
H^2(\mathcal{E}|_{L^2}(-2))=0
$
since $H^q(\mathcal{E}|_{L}(-1))=0$ for all $q>0$. 
Summing up, we have 
\begin{equation}\label{H^2vanishing}
H^2(\mathcal{E}|_{L^2}(-k))=0
\end{equation}
for all $k\leq 2$ and any plane $L^2$ in $\mathbb{P}^n$.

If $n=2$, the Riemann-Roch formula for a twisted vector bundle $\mathcal{E}(t)$
is
\begin{equation}\label{RRonP2}
\chi(\mathcal{E}(t))=\dfrac{1}{2}(rt+6)(t+3)+r-c_2.
\end{equation}
Recall that this formula is for $c_1=c_1(\mathcal{E})=3$.
It follows from this formula that 
$\chi(\mathcal{E}(-2))=3-c_2$.
For arbitrary $n\geq 2$,
the vanishing (\ref{H^2vanishing})
then implies that 
\begin{equation}\label{chiE-2}
h^0(\mathcal{E}|_{L^2}(-2))-h^1(\mathcal{E}|_{L^2}(-2))=3-c_2(\mathcal{E}|_{L^2})
\end{equation}
for any plane $L^2$ in $\mathbb{P}^n$.

We claim here that
\[
c_2
\geq 3
\] on $\mathbb{P}^n$.
Suppose, to the contrary, that 
$c_2
\leq 2$.
Then $\chi(\mathcal{E}|_{L^2}(-2))=3-c_2(\mathcal{E}|_{L^2})\geq 1$
since 
$c_2
=c_2(\mathcal{E}|_{L^2})$.
Hence we obtain $h^0(\mathcal{E}|_{L^2}(-2))\neq 0$.
As we have seen, this implies that $\mathcal{E}|_{L^2}$ lies 
either in the case (1), (2), or (5) of Theorem~\ref{c_1=3c_2<8};
if $\mathcal{E}|_{L^2}$ lies in the case (5)
then $c_2(\mathcal{E}|_{L^2})=3$,
which
contradicts that 
$c_2
\leq 2$.
Thus $\mathcal{E}|_{L^2}$ actually lies either in the case (1) or (2) of Theorem~\ref{c_1=3c_2<8}.
In particular it 
splits 
into a direct sum of line bundles.
Hence $\mathcal{E}$ 
also 
splits
by the splitting criterion 
\cite[Theorem 2.3.2]{oss} 
of Horrocks.
Since $h^0(\mathcal{E}|_{L^2}(-2))\neq 0$, this implies that $h^0(\mathcal{E}(-2))\neq 0$,
which contradicts the assumption~(\ref{first assumption}). Hence the claim follows.

Since $\mathcal{E}|_{L^2}$ is nef, 
we have 
an inequality
$0\leq H(\mathcal{E}|_{L^2})^{r+1}$.
Since $H(\mathcal{E}|_{L^2})^{r+1}$
is equal to 
$c_1(\mathcal{E}|_{L^2})^2-c_2(\mathcal{E}|_{L^2})$
(see \cite[\S 3.1 and \S 3.2]{fl}),
we obtain $c_2(\mathcal{E}|_{L^2})\leq 9$.
Hence we 
have 
\[
c_2
\leq 9.
\]
Recall here the well-known non-negativity (see, e.g., \cite[Theorem 8.2.1]{MR2095472})
of the top Chern class of a nef vector bundle $\mathcal{E}$
that 
\begin{equation}\label{topnonnegative}
c_n
\geq 0.
\end{equation}

We shall divide the proof of Theorem~\ref{c_1=3c_2<8} according to the value of 
$c_2
\geq 3
$.

Note that $H(\mathcal{E}|_{L^2})$ is nef and big
if $c_2< 9$.
The vanishing 
(\ref{KVvanishing in arbitrary dim})
then implies that 
\begin{equation}\label{KVvanishing}
H^q(\mathcal{E}|_{L^2})=0
\end{equation}
for any $q>0$ and any plane $L^2$ in $\mathbb{P}^n$.
Together with the vanishing~(\ref{first vanishing}),
this implies that 
\begin{equation}\label{H2ijouvanishingforE(-1)onP3}
H^q(\mathcal{E}|_{L^3}(-1))=0
\end{equation}
for any $q\geq 2$ and any three-dimensional linear subspace $L^3\subset\mathbb{P}^n$.

\subsection{Set-up for the two-dimensional case with $h^1(\mathcal{E})=0$}\label{Set-up for the two-dimensional case}
Note that the vanishing $h^1(\mathcal{E})=0$ holds if 
$c_2<9$ by (\ref{KVvanishing}).
It follows from (\ref{first assumption}) and (\ref{chiE-2}) that 
\begin{equation}\label{h1E-2onP2}
h^1(\mathcal{E}(-2))=c_2-3.
\end{equation}
The Riemann-Roch formula (\ref{RRonP2}) 
shows 
that 
$\chi(\mathcal{E}(-1))=6-c_2$
and $\chi(\mathcal{E})=9+r-c_2$.
Since we have the vanishings (\ref{H^2vanishing}) and $h^1(\mathcal{E})=0$, these formulas imply
\begin{gather}
h^0(\mathcal{E}(-1))-h^1(\mathcal{E}(-1))=6-c_2,\label{RRonP2(-1)}\\
h^0(\mathcal{E})=9+r-c_2.\label{RRonP2(0)}
\end{gather}
Since we have an exact sequence
$
0\to \mathcal{E}(-2)\to \mathcal{E}(-1)\to \mathcal{E}|_{L}(-1)\to 0,
$
by taking cohomology, 
we see that 
\begin{equation}\label{h1E-2biggerthanh1E-1}
h^1(\mathcal{E}(-2))
\geq h^1(\mathcal{E}(-1)).
\end{equation}

Now
we apply to $\mathcal{E}$ the Bondal spectral sequence (\ref{BondalSpectralSequence}).
It is clear that $E_2^{p,q}=0$ if $q< 0$ or $p>0$.
The vanishing~(\ref{H^2vanishing}) shows that $E_2^{p,q}=0$ 
if
$q\geq 2$.
Since $H^1(\mathcal{E})= 0$ by assumption,
the right $A$-module $\Ext^1(G,\mathcal{E})$ fits in 
an exact sequence
\[
0\to S_1^{\oplus h^1(\mathcal{E}(-1))}\to \Ext^1(G,\mathcal{E})\to S_2^{\oplus 
h^1(\mathcal{E}(-2))
}\to 0.
\]
Since $S_1\lotimes_AG\cong \Omega_{\mathbb{P}^2}(1)[1]$
and $S_2\lotimes_AG\cong \mathcal{O}(-1)[2]$ by (\ref{Sjpullback}), 
the sequence above induces the following distinguished triangle
in $D^b(X)$:
\[\mathcal{O}(-1)^{\oplus h^1(\mathcal{E}(-2))}[1] 
\to \Omega_{\mathbb{P}^2}(1)^{\oplus h^1(\mathcal{E}(-1))}[1]\to \Ext^1(G,\mathcal{E})\lotimes_AG\to.
\]
Since $E_2^{p,1}=\mathcal{H}^{p}(\Ext^1(G,\mathcal{E})\lotimes_AG)$, the triangle above 
shows that $E_2^{p,1}=0$ unless $p=-2$ or $-1$ and that $E_2^{-2,1}$ and $E_2^{-1,1}$ fit in 
the following 
exact sequence of coherent sheaves:
\begin{equation}\label{exactseq}
0\to E_2^{-2,1}\to 
\mathcal{O}(-1)^{\oplus h^1(\mathcal{E}(-2))} 
\xrightarrow{\mu} \Omega_{\mathbb{P}^2}(1)^{\oplus h^1(\mathcal{E}(-1))}
\to E_2^{-1,1}\to 0.
\end{equation}
It follows from (\ref{first assumption}) that the right $A$-module $\Hom(G,\mathcal{E})$ has, as in (\ref{typicalprojresol}),
a projective resolution of the form
\[
0\to P_0^{\oplus 3e_{0,1}}\to P_1^{\oplus e_{0,1}}\oplus P_0^{\oplus e_{0,0}}\to \Hom (G,\mathcal{E})\to 0
\]
where $e_{0,0}=h^0(\mathcal{E})$ and $e_{0,1}=h^0(\mathcal{E}(-1))$.
In particular, we see that $E_2^{p,0}=0$ if $p<-1$.
We have the following exact sequence
\begin{equation}\label{E_3^{0,0}definition}
0\to E_2^{-2,1}\to E_2^{0,0}\to E_3^{0,0}\to 0.
\end{equation}
Now the Bondal spectral sequence implies 
that $E_2^{-1,0}=0$ and 
that $\mathcal{E}$
fits in an exact sequence
\begin{equation}\label{E_2^{-1,1}quotient}
0\to E_{3}^{0,0}\to \mathcal{E}\to E_2^{-1,1}\to 0.
\end{equation}
Since $E_2^{-1,0}=0$, $E_2^{0,0}$ fits in an exact sequence
\begin{equation}\label{E_2^00 exact sequence in dim 2}
0\to \mathcal{O}^{\oplus 3e_{0,1}}\to \mathcal{O}(1)^{\oplus e_{0,1}}\oplus \mathcal{O}^{\oplus e_{0,0}}\to E_2^{0,0}\to 0.
\end{equation}
The following lemma shall be applied repeatedly throughout this paper.
\begin{lemma}\label{very often}
For any finite morphism $C\to \mathbb{P}^2$ from a smooth projective curve $C$,
the pullback $E_2^{-1,1}|_C$ of the sheaf $E_2^{-1,1}$
can not admit a line bundle of negative degree as a quotient.
In particular, $E_2^{-1,1}$ can not admit the following sheaves as a quotient:
\begin{enumerate}
\item[$(1)$] $\Omega_{\mathbb{P}^2}(1)$; 
\item[$(2)$] $\mathcal{I}_p$, where $p$ is a point;
\item[$(3)$] $\mathcal{I}_Z(1)$, where $Z$ is a $0$-dimensional closed subscheme of $\length Z\geq 2$;
\item[$(4)$] $\mathcal{I}_Z(2)$, where $Z$ is a $0$-dimensional closed subscheme of $\length Z\geq 5$;
\item[$(5)$] $\mathcal{O}_L(-p)$, where $L$ is a line passing through a point $p$;
\item[$(6)$] $\mathcal{O}_C(-p)$, where $C$ is a conic passing through a point $p$.
\end{enumerate}
\end{lemma}
\begin{proof}
The first statement follows from the sequence (\ref{E_2^{-1,1}quotient}) 
since $\mathcal{E}$ is nef.
The second statement is almost obvious from the first, so that we only give a proof in case $(4)$.
If there exists a line $L$ such that $\length Z\cap L\geq 3$, then 
the double twist $\mathcal{I}_Z\cdot\mathcal{O}_L(2)$
of the inverse image ideal sheaf $\mathcal{I}_Z\cdot\mathcal{O}_L$
has 
negative degree, which contradicts the first statement.
If $\length Z\cap L\leq 2$ for any line $L$, then Lemma~\ref{conicpassing5points} below
shows that there exists a smooth conic $C$ such that $\length Z\cap C\geq 5$.
Hence 
$\mathcal{I}_Z\cdot\mathcal{O}_C(2)$
has 
negative degree, which contradicts the first statement again.
Therefore $E_2^{-1,1}$ can not admit $\mathcal{I}_Z(2)$ as a quotient if $\length Z\geq 5$.
\end{proof}
Lemma~\ref{very often} and the exact sequence~(\ref{exactseq}) indicate that there are some relations between 
$h^1(\mathcal{E}(-2))$ 
and $h^1(\mathcal{E}(-1))$;
we shall explore these relations in Lemma~\ref{key}.

\subsection{Set-up for the three-dimensional case}
The Riemann-Roch formula for a twisted vector bundle $\mathcal{E}(t)$ 
on $\mathbb{P}^3$
is
\[\chi(\mathcal{E}(t))=\frac{1}{2}\{9-3c_2+c_3+
(9-2c_2)(t+2)+(3t^2+12t+11)\}+\frac{r}{6}(t+3)(t+2)(t+1).\]
Recall that this formula is for $c_1=c_1(\mathcal{E})=3$.
By this formula,
we have 
\begin{gather}
\chi(\mathcal{E}(-3))=1-\frac{1}{2}(c_2-c_3),\label{RRonP3(-3)}\\
\chi(\mathcal{E}(-2))=4-\frac{1}{2}(3c_2-c_3),\label{RRonP3(-2)}\\
\chi(\mathcal{E}(-1))=10-\frac{1}{2}(5c_2-c_3).\label{RRonP3(-1)}
\end{gather}
By taking into account 
the vanishing~(\ref{first vanishing}), we also have 
\begin{gather}
h^0(\mathcal{E})
=19-\frac{1}{2}(7c_2-c_3)+r,\label{RRonP3(0)}\\
h^0(\mathcal{E}(1))
=31-\frac{1}{2}(9c_2-c_3)+4r.\label{RRonP3(1)}
\end{gather}
Recall here that 
$H(\mathcal{E})^{r+2}
=c_3-2c_1c_2+c_1^3$
(see \cite[\S 3.1 and \S 3.2]{fl}).
Since $c_1=3$ and 
$H(\mathcal{E})^{r+2}\geq 0$, 
this implies that 
\begin{equation}\label{selfintersection}
c_3\geq 6c_2-27.
\end{equation}

\section{Lemmas}\label{Lemmas}
In this section, we collect some lemmas applied 
in the proofs of Theorems~\ref{c_1=3c_2<8}
or \ref{globalgeneration}.
\begin{lemma}\label{zero locus of T(-1)}
Let $V$ be an $(n+1)$-dimensional vector space,
$(x_0,x_1,\dots,x_n)$ a basis of $V$,
and $(x_0^*,\dots,x_n^*)$ the basis of the dual vector space $V^*$
corresponding to $(x_0,x_1,\dots,x_n)$.
Suppose that a non-zero element $s$ of $H^0(T_{\mathbb{P}(V)}(-1))$
corresponds to $\sum_{i=0}^{n}a_ix_i^*$ for some $a_i\in K$ via the isomorphism
from $H^0(T_{\mathbb{P}(V)}(-1))$ to $V^*$.
Then the zero locus $(s)_0$ of $s$ is a (reduced) point $p$
whose homogeneous coordinates $(x_0(p)\colon\cdots\colon x_n(p))$ are $(a_0\colon a_1\colon \cdots\colon a_n)$.
\end{lemma}
\begin{proof}
See \cite[Remark 2.3]{sato}.
\end{proof}
\begin{lemma}\label{BeilinsonGreat}
Let $i$ be any integer such that $1\leq i\leq n$.
For each non-zero morphism $\varphi:\Omega_{\mathbb{P}(V)}^i(i)\to \Omega_{\mathbb{P}(V)}^{i-1}(i-1)$,
there exists a unique element $s\in H^0(T_{\mathbb{P}(V)}(-1))$
such that $\varphi$ is the morphism appearing in the Koszul complex induced from $s$.
In particular, every non-zero morphism 
$\varphi:
\Omega_{\mathbb{P}^3}^2(2)\to \Omega_{\mathbb{P}^3}(1)$
determines a unique non-zero element $s\in H^0(T_{\mathbb{P}^3}(-1))$
whose corresponding Koszul complex is an exact sequence
\[
0\to \mathcal{O}(-1)\to \Omega_{\mathbb{P}^3}^2(2)\xrightarrow{\varphi} \Omega_{\mathbb{P}^3}(1)
\to \mathcal{I}_p
\to 0,
\]
where $\{p\}=(s)_0$.
\end{lemma}
\begin{proof}
We have an isomorphism $H^0(T_{\mathbb{P}(V)}(-1))\cong \Hom(\Omega_{\mathbb{P}(V)}^i(i), \Omega_{\mathbb{P}(V)}^{i-1}(i-1))$
by \cite[Lemma 2]{MR0509388}, and this isomorphism indeed sends a global section $s$ of $T_{\mathbb{P}(V)}(-1)$ to 
the morphism $\varphi:\Omega_{\mathbb{P}(V)}^i(i)\to \Omega_{\mathbb{P}(V)}^{i-1}(i-1)$ appearing in the Koszul complex associated to $s$.
Note that a non-zero section $s$ is regular
by Lemma~\ref{zero locus of T(-1)}, so that the corresponding Koszul complex is exact.
\end{proof}
\begin{lemma}\label{DhaO}
Let $W$ be a $0$-dimensional closed subscheme of $\mathbb{P}^2$
and $w$ a point in $\mathbb{P}^2$.
Suppose that we have the following non-split exact sequence of coherent sheaves
\[
0\to \mathcal{I}_W(d)\to \mathcal{D}\to k(w)\to 0,
\]
where $d$ is an integer
and $k(w)$  denotes the residue field of the point $w$.
Then $w$ is an associated point of $W$
and $\mathcal{D}$ is isomorphic to $\mathcal{I}_{Z}(d)$
where $Z$ is a closed subscheme of $W$ with $\length Z=\length W-1$.
\end{lemma}
\begin{proof}
First note that $w$ is not an associated point of $\mathcal{D}$,
since the exact sequence above does not split.
Hence $\mathcal{D}$ is a torsion-free sheaf of rank one, and its double dual is isomorphic to $\mathcal{O}(d)$.
Therefore $\mathcal{D}$ is isomorphic to $\mathcal{I}_{Z}(d)$,
where $Z$ is a closed subscheme of $W$ with $\length Z=\length W-1$,
and $w$ is an associated point of $W$.
\end{proof}
\begin{lemma}\label{conicpassing5points}
Let $Z$ be a $0$-dimensional closed subscheme of length 
$5$
in $\mathbb{P}^2$,
and suppose that $\length(Z\cap L)\leq 2$ for any line $L$ in $\mathbb{P}^2$.
Then there exists a smooth conic $C$ 
passing through $Z$.
\end{lemma}
\begin{proof}
See, e.g., the second paragraph of page 75 of \cite{Nefofc1=3c2=8OnPN}.
\end{proof}

\section{Key Lemma
}\label{KeyLemma}
The following lemma together with the exact sequence (\ref{E_2^{-1,1}quotient})
plays a crucial role in the proofs of Theorems~\ref{c_1=3c_2<8} and \ref{globalgeneration}.
\begin{lemma}\label{key}
Let 
\[
0\to E_2^{-2,1}\to \mathcal{O}(-1)^{\oplus a} 
\xrightarrow{\mu} \Omega_{\mathbb{P}^2}(1)^{\oplus b}
\to E_2^{-1,1}\to 0
\]
be the exact sequence 
(\ref{exactseq}) on $\mathbb{P}^2$,
where $a=h^1(\mathcal{E}(-2))$ and  $b=h^1(\mathcal{E}(-1))$.
\begin{enumerate}
\item[$(1)$] If $b\geq 1$, consider a 
surjection
$\pi:\Omega_{\mathbb{P}^2}(1)^{\oplus b}\to \Omega_{\mathbb{P}^2}(1)$,
and 
let 
\[\varphi:
\mathcal{O}(-1)^{\oplus a} 
\to \Omega_{\mathbb{P}^2}(1)
\] 
be the composite of $\mu$ and $\pi$.
Then $H^0(\varphi(1)):H^0(\mathcal{O}^{\oplus a})\to H^0(\Omega_{\mathbb{P}^2}(2))$ is surjective;
consequently 
$a\geq 3$ and 
$\varphi$ is surjective.
Moreover we have an exact sequence
\[
0\to E_2^{-2,1}\to \mathcal{O}(-2)\oplus \mathcal{O}(-1)^{\oplus a-3} 
\xrightarrow{\mu_1} \Omega_{\mathbb{P}^2}(1)^{\oplus b-1}
\to E_2^{-1,1}\to 0.
\]
\item[$(2)$] 
If $b\geq 2$, 
consider a surjection
$q:\Omega_{\mathbb{P}^2}(1)^{\oplus b-1}\to \Omega_{\mathbb{P}^2}(1)$,
and let 
\[
\varphi_1:
\mathcal{O}(-2)\oplus\mathcal{O}(-1)^{\oplus a-3} 
\to \Omega_{\mathbb{P}^2}(1)
\]
be the composite of $\mu_1$ and $q$.
Then the image of $H^0(\varphi_1(1))$ has dimension two or three;
consequently
$a\geq 5$.
\begin{enumerate}
\item[$(a)$] If $\dim \im H^0(\varphi_1(1))=2$, then we have the following exact sequence
\[
\begin{split}
0\to E_2^{-2,1}\to \mathcal{O}(-3)\oplus\mathcal{O}(-1)^{\oplus a-5}
&\xrightarrow{\nu_2}
\Omega_{\mathbb{P}^2}(1)^{\oplus b-2}\\
&\to E_2^{-1,1}\to k(w)\to 0
\end{split}
\]
for some point $w$ in $\mathbb{P}^2$,
where $k(w)$ denotes the residue field of $w$.
\item[$(b)$] If $\dim \im H^0(\varphi_1(1))=3$, then $a\geq 6$ and we have the following exact sequence
\[0\to E_2^{-2,1}\to \mathcal{O}(-2)^{\oplus 2}\oplus\mathcal{O}(-1)^{\oplus a-6}
\xrightarrow{\mu_2}
\Omega_{\mathbb{P}^2}(1)^{\oplus b-2}
\to E_2^{-1,1}\to 0.
\]
\end{enumerate}
\end{enumerate}
\end{lemma}
\begin{proof}
(1) First note that there exists a surjection $E_2^{-1,1}\to \Coker\varphi$.

If $H^0(\varphi(1))=0$, then $\varphi=0$ and $\Coker\varphi\cong \Omega_{\mathbb{P}^2}(1)$.
This however contradicts Lemma~\ref{very often}.
Therefore $H^0(\varphi(1))\neq 0$.

Suppose that the image of $H^0(\varphi(1))$ has dimension one,
and let $s$ be a non-zero element in the image in $H^0(\Omega_{\mathbb{P}^2}(2))\cong H^0(T_{\mathbb{P}^2}(-1))$.
Then the zero locus of $s$ is a (reduced) point $p$ by Lemma~\ref{zero locus of T(-1)}
and $\Coker\varphi$ is isomorphic to the ideal sheaf $\mathcal{I}_p$
of $p$,
but this can not occur by Lemma~\ref{very often}.
Therefore the image of $H^0(\varphi(1))$ has dimension $\geq 2$.

Suppose that the image of $H^0(\varphi(1))$ has dimension two,
and let $(s, t)$ be a basis of the image.
As above, let $p$ be the zero locus of $s$.
Then $t$ induces an injection $\mathcal{O}\to \mathcal{I}_p(1)$,
and this gives a line $L$ passing through $p$.
The quotient of the $\mathcal{O}(-1)$-twist of this injection is $\mathcal{O}_{L}(-p)$ and is isomorphic to $\Coker\varphi$.
However this is impossible by Lemma~\ref{very often}.
Therefore the image of $H^0(\varphi(1))$ has dimension $\geq 3$, i.e., $H^0(\varphi(1))$ is surjective.
Hence $\varphi$ is surjective and $\Ker \varphi\cong \mathcal{O}(-2)\oplus\mathcal{O}(-1)^{\oplus a-3}$.
Now $\Ker \pi\cong \Omega_{\mathbb{P}^2}(1)^{\oplus b-1}$, and the morphism $\mu_1:\Ker \varphi\to \Ker \pi$
induced by $\mu$ 
extends to the exact sequence in the statement.

(2) Note first that there exists a surjection $E_2^{-1,1}\to \Coker\varphi_1$.
Denote by $\iota$ the inclusion $\mathcal{O}(-1)^{\oplus a-3}\to \mathcal{O}(-2)\oplus \mathcal{O}(-1)^{\oplus a-3}$,
and consider the composite $\varphi_1\circ \iota:\mathcal{O}(-1)^{\oplus a-3}\to \Omega_{\mathbb{P}^2}(1)$.
We see that $\varphi_1$ induces a morphism 
$\bar{\varphi}_1:
\Coker(\iota)
\to \Coker(\varphi_1\circ \iota)$
and that $\Coker(\iota)\cong \mathcal{O}(-2)$.
Moreover we have the following long exact sequence 
\begin{equation}\label{long exact in key lemma}
0\to \Ker(\varphi_1\circ\iota)\to \Ker\varphi_1\to \mathcal{O}(-2)\xrightarrow{\bar{\varphi}_1}
\Coker(\varphi_1\circ\iota)\to \Coker\varphi_1\to 0
\end{equation}
by the snake lemma.
In particular,
$\Coker \varphi_1\cong \Coker \bar{\varphi}_1$.

If $H^0(\varphi_1(1))=0$, then 
$\varphi_1\circ\iota=0$ and thus $\Coker(\varphi_1\circ \iota)\cong \Omega_{\mathbb{P}^2}(1)$.
If $\bar{\varphi}_1=0$, then 
$\Coker
\varphi_1
\cong 
\Coker(\varphi_1\circ \iota)
$.
Hence 
$E_2^{-1,1}$ has $\Omega_{\mathbb{P}^2}(1)$ as a quotient,
which 
is absurd by Lemma~\ref{very often}.
Therefore
$\bar{\varphi}_1
:\mathcal{O}(-2)\to \Omega_{\mathbb{P}^2}(1)$
is a non-zero morphism. 
Suppose that 
$\bar{\varphi}_1$ 
is decomposed as 
$\bar{\varphi}_1=sl$, 
where 
$l$ is 
an inclusion $\mathcal{O}(-2)\to\mathcal{O}(-1)$
and 
$s$ is 
a non-zero
morphism $\mathcal{O}(-1)\to \Omega_{\mathbb{P}^2}(1)$.
The cokernel of  the inclusion $l:\mathcal{O}(-2)\to\mathcal{O}(-1)$
is isomorphic to $\mathcal{O}_{L}(-1)$, where $L$ is the line defined by $l$,
and the cokernel of the morphism $s:\mathcal{O}(-1)\to \Omega_{\mathbb{P}^2}(1)$
is isomorphic to $\mathcal{I}_p$, where 
$p$ is a point;
thus $\Coker\bar{\varphi}_1$ fits in an exact sequence
\[0\to \mathcal{O}_{L}(-1)\to \Coker\bar{\varphi}_1\to \mathcal{I}_p\to 0.\]
Therefore
$\Coker\varphi_1$ (and hence $E_2^{-1,1}$) has $\mathcal{I}_p$ as a quotient,
which can not happen by Lemma~\ref{very often}.
Hence 
$\bar{\varphi}_1$
does not factor 
as 
$\bar{\varphi}_1=sl$.
Let $\bar{s}$ be the non-zero section of $H^0(\Omega_{\mathbb{P}^2}(3))$ 
determined by $\bar{\varphi}_1$.
Then $\bar{s}$ is a regular section of $\Omega_{\mathbb{P}^2}(3)\cong T_{\mathbb{P}^2}$,
i.e., the zero locus $Z$ of $\bar{s}$ has dimension $\leq 0$.
Since $c_2(T_{\mathbb{P}^2})=3$, 
$Z$ 
is thereby a $0$-dimensional closed subscheme of length three.
Note that 
$\Coker\bar{\varphi}_1
\cong 
\mathcal{I}_Z(1)$;
this however contradicts Lemma~\ref{very often}.
Therefore $H^0(\varphi_1(1))\neq 0$.

Suppose that the image of $H^0(\varphi_1(1))$ has dimension one.
Then $\Coker(\varphi_1\circ\iota)$ is isomorphic to the ideal sheaf $\mathcal{I}_p$
of some point $p$, as is shown in the proof of (1).
If $\bar{\varphi}_1=0$, then 
$\Coker\bar{\varphi}_1\cong \mathcal{I}_p$,
which is absurd by Lemma~\ref{very often}.
Hence $\bar{\varphi}_1:\mathcal{O}(-2)\to \mathcal{I}_p$ is non-zero.
Then $\Coker\bar{\varphi}_1$ is isomorphic to $\mathcal{O}_C(-p)$
for some conic $C$ on $\mathbb{P}^2$,
which again
contradicts Lemma~\ref{very often}.
Therefore the image of $H^0(\varphi_1(1))$ has dimension at least two.

(2) $(a)$ 
Suppose that 
$\dim \im H^0(\varphi_1(1))=2$.
Then $\Coker(\varphi_1\circ\iota)
\cong 
\mathcal{O}_{L}(-p)$
for some line $L$ in $\mathbb{P}^2$ and some point $p$ in $L$, as is shown in the proof of (1).
If $\bar{\varphi}_1=0$, then 
$\Coker\bar{\varphi}_1\cong \mathcal{O}_{L}(-p)$,
which can not happen by Lemma~\ref{very often}.
Hence $\bar{\varphi}_1:\mathcal{O}(-2)\to \mathcal{O}_{L}(-p)$ is non-zero.
Then $\Coker\bar{\varphi}_1\cong k(w)$
for some point $w$ on $\mathbb{P}^2$,
and $\Ker \bar{\varphi}_1\cong \mathcal{O}(-3)$.
The long exact sequence~(\ref{long exact in key lemma}) induces an exact sequence
\[0\to \Ker(\varphi_1\circ\iota)\to \Ker\varphi_1\to \Ker\bar{\varphi}_1\to 0.\]
Note here that
$\Ker \varphi_1\circ\iota\cong \mathcal{O}(-1)^{\oplus a-5}$;
the above sequence then implies that
$\Ker \varphi_1\cong \mathcal{O}(-3)\oplus \mathcal{O}(-1)^{\oplus a-5}$.
Now we get the exact sequence in the statement by the snake lemma.

(2) $(b)$
Suppose that 
$\dim \im H^0(\varphi_1(1))\geq 3$.
Then $H^0(\varphi_1(1))$ is surjective,
and thus $\varphi_1\circ\iota$ is surjective.
Hence $\Ker \varphi_1\circ\iota\cong \mathcal{O}(-2)\oplus \mathcal{O}(-1)^{\oplus a-6}$.
Then the long exact sequence~(\ref{long exact in key lemma}) gives an exact sequence
\[0\to \mathcal{O}(-2)\oplus \mathcal{O}(-1)^{\oplus a-6}\to \Ker \varphi_1\to \mathcal{O}(-2)\to 0.\]
Therefore 
$\Ker \varphi_1\cong \mathcal{O}(-2)^{\oplus 2}\oplus \mathcal{O}(-1)^{\oplus a-6}$.
Now $\Ker q\cong \Omega_{\mathbb{P}^2}(1)^{\oplus b-2}$, and $\mu_1$ induces a morphism $\Ker\varphi_1\to \Ker q$, which extends 
to the exact sequence in the statement.
\end{proof}

\section{Global generation of $\mathcal{E}$ in case $c_2\leq 7$}\label{shortcut}
\begin{thm}\label{globalgeneration}
Let $\mathcal{E}$ be a nef vector bundle of rank $r$ on a projective space $\mathbb{P}^n$
with first Chern class $c_1=3$ and second Chern class $c_2\leq 7$.
Then $\mathcal{E}$ is globally generated.
\end{thm}
\begin{proof}
Consider first the case $n=2$. Note that $\mathcal{E}$ is $0$-regular in the sense of Castelnuovo-Mumford
if $h^1(\mathcal{E}(-1))=0$ since we have the vanishing (\ref{H^2vanishing}).
Since $0$-regularity implies global generation, we may concentrate our attention to the case 
$h^1(\mathcal{E}(-1))\neq 0$.
We shall work in the setting in \S\ref{Set-up for the two-dimensional case}
to apply the Bondal spectral sequence (\ref{BondalSpectralSequence}).
Then it follows from Lemma~\ref{key} (1)
and formula (\ref{h1E-2onP2}) that $c_2\geq 6$.
Moreover the assumption $c_2\leq 7$ implies that 
\begin{equation}\label{n=2c2=67h1=1}
h^1(\mathcal{E}(-1))=1
\end{equation} 
by Lemma~\ref{key} (2).
Lemma~\ref{key} (1) 
then shows that $E_2^{-1,1}=0$. 
Therefore $E_3^{0,0}\cong \mathcal{E}$ by the exact sequence (\ref{E_2^{-1,1}quotient}).
Note that $E_3^{0,0}$ is a quotient of $E_2^{0,0}$ since $E_2^{p,q}=0$ if $p>0$.
Since $E_2^{0,0}$ is globally generated by (\ref{E_2^00 exact sequence in dim 2}),
this shows that $\mathcal{E}$ is globally generated if $n=2$ and $c_2\leq 7$.

Consider next the case $n\geq 3$.
Recall here that $\mathcal{E}$ is globally generated 
if $h^1(\mathcal{E}(-1))=0$
and $\mathcal{E}|_H$ is globally generated for any hyperplane $H$ in $\mathbb{P}^n$
by \cite[Lemma 3]{pswnef}.

Suppose now that $n=3$. 
Since $\mathcal{E}|_H$ is globally generated for any plane $H$ 
as we have seen above,
we may concentrate our attention to the case $h^1(\mathcal{E}(-1))\neq 0$.
Then 
it follows from (\ref{KVvanishing in arbitrary dim})
that $H(\mathcal{E})$ is not big, and hence $c_3=6c_2-27$ by (\ref{selfintersection}).
Thus $c_3$ is odd. Hence $c_2$ is odd  by (\ref{RRonP3(-1)}).
Since $c_3$ is non-negative by (\ref{topnonnegative}), 
we see that $c_2\geq 5$ by equality $c_3=6c_2-27$;
thus $(c_2,c_3)=(5,3)$ or $(7,15)$. 
Suppose that $(c_2,c_3)=(7,15)$.
Then $h^1(\mathcal{E}|_H(-1))=1$ by (\ref{n=2c2=67h1=1}).
Hence $h^0(\mathcal{E}|_H(-1))=0$ by (\ref{RRonP2(-1)}).
Since $h^0(\mathcal{E}(-2))=0$ by (\ref{first assumption}),
this implies that $h^0(\mathcal{E}(-1))=0$.
Note here that $H^q(\mathcal{E})=0$ for all $q>0$
by 
(\ref{first vanishing})
and that $H^q(\mathcal{E}|_H)=0$ for all $q>0$
by (\ref{KVvanishing}).
Hence 
$H^q(\mathcal{E}(-1))=0$ unless $q=1$.
Thus 
\[-h^1(\mathcal{E}(-1))=\chi(\mathcal{E}(-1))=(c_3-15)/2=0\]
by 
(\ref{RRonP3(-1)}).
This contradicts the assumption $h^1(\mathcal{E}(-1))\neq 0$;
thus the case where $(c_2,c_3)=(7,15)$ is ruled out.
In \S \ref{c2=5case}, we shall show that the case $(c_2,c_3)=(5,3)$ does not happen.
Therefore $h^1(\mathcal{E}(-1))=0$ and $\mathcal{E}$ is globally generated.

If $n\geq 4$,
then $h^1(\mathcal{E}(-1))=0$ by (\ref{first vanishing})
and $\mathcal{E}|_H$ is globally generated for any hyperplane $H$ in $\mathbb{P}^n$
as we have seen above.
Therefore $\mathcal{E}$ is globally generated.
\end{proof}

\subsection{The case where $(n,c_2, c_3)=(3,5,3)$}\label{c2=5case}
We first claim that 
\begin{eqnarray}
H^q(\mathcal{E}|_{L^2}(-k))&=0&
\textrm{ for all }q>0
\textrm{ and all }k\leq 2
\textrm{ unless }
(q,k)=(1,2);\label{restrictionVanishing(n,c_2, c_3)=(3,5,3)}
\\
h^1(\mathcal{E}|_{L^2}(-2))&=2,& 
\end{eqnarray}
where $L^2$ denotes a plane in $\mathbb{P}^3$.
Indeed we have $h^1(\mathcal{E}|_{L^2}(-2))=2$
by (\ref{h1E-2onP2}),
and 
$H^q(\mathcal{E}|_{L^2}(-k))$ vanishes
for any 
$k\leq 0$ and 
any
$q>0$ by (\ref{KVvanishing}),
for $(k,q)=(1,2)$ and $(2,2)$ by (\ref{H^2vanishing}),
and for $(k,q)=(1,1)$ by Lemma~\ref{key} (1)
and (\ref{h1E-2onP2}).
Since $H^q(\mathcal{E}(-k))=0$ for all $q>0$ if  $k\leq 0$ by (\ref{first vanishing}),
the claims above imply that 
\begin{align}
H^q(\mathcal{E}(-k))&=0
\textrm{ for all }q\geq 2
\textrm{ and all }k\leq 3
\textrm{ unless }(q,k)=(2,3);\label{H2ijouvanishingincasen=3c2=5}\\
h^2(\mathcal{E}(-3))&\leq 2.\label{aleq2}
\end{align}
Set 
\begin{equation}\label{definitionOfa}
a=h^2(\mathcal{E}(-3)).
\end{equation}
Then 
\begin{equation}\label{h1E-3haA}
a=h^1(\mathcal{E}(-3))
\end{equation}
by (\ref{first assumption}),
(\ref{RRonP3(-3)}),
and (\ref{H2ijouvanishingincasen=3c2=5}).
It follows from 
(\ref{first assumption}), (\ref{RRonP3(-2)}), and 
(\ref{H2ijouvanishingincasen=3c2=5})
that 
\begin{equation}\label{c2=5c3=3h1E(-2)}
h^1(\mathcal{E}(-2))=2.
\end{equation}
We also have 
\[
\chi(\mathcal{E}(-1))
=-1
\]
by 
(\ref{RRonP3(-1)}).
By 
(\ref{RRonP3(0)}),
we have 
\begin{equation}\label{h0E(0)andE(1)incasec2=5c3=3}
h^0(\mathcal{E})=r+3.
\end{equation}
Note that there exists an
exact sequence
\[
0\to H^0(\mathcal{E}(-1))\to H^0(\mathcal{E}|_{L^2}(-1))\to H^1(\mathcal{E}(-2))\to H^1(\mathcal{E}(-1))\to 0
\]
and that $h^0(\mathcal{E}|_{L^2}(-1))=1$ by (\ref{RRonP2(-1)}) and (\ref{restrictionVanishing(n,c_2, c_3)=(3,5,3)}).
Therefore we have two cases, corresponding to $h^0(\mathcal{E}(-1))=0$ or $1$. If 
$h^0(\mathcal{E}(-1))=1$, then $h^1(\mathcal{E}(-1))=2$, and 
we shall consider the following exact sequence 
\begin{equation}\label{defofF}
0\to \mathcal{O}(1)\to \mathcal{E}\to \mathcal{F}\to 0
\end{equation}
of coherent sheaves on $\mathbb{P}^3$.
Note here that $\Ext^q(G,\mathcal{E})\cong \Ext^q(G,\mathcal{F})$ if $q>0$,
that $\Hom(\mathcal{O}(i),\mathcal{F})=0$ if $i\geq 2$,
that $\Hom(\mathcal{O}(1),\mathcal{F})=0$ by the assumption $h^0(\mathcal{E}(-1))=1$,
and that $\hom(\mathcal{O},\mathcal{F})=r-1$ by (\ref{h0E(0)andE(1)incasec2=5c3=3}).

\subsubsection{Set-up using the Bondal spectral sequence}\label{SetUpFor(n,c_2, c_3)=(3,5,3)}
Suppose that $(h^0(\mathcal{E}(-1)), h^1(\mathcal{E}(-1)))=(0,1)$ (resp.\ $(1,2)$).
We apply to $\mathcal{E}$ (resp.\ $\mathcal{F}$) the Bondal spectral sequence 
(\ref{BondalSpectralSequence}).
It follows from (\ref{H2ijouvanishingincasen=3c2=5}) that $E_2^{p,q}=0$ for all $p$ if $q\geq 3$.
We have $\Ext^2(G,\mathcal{E})\cong S_3^{\oplus a}$ by 
(\ref{H2ijouvanishingincasen=3c2=5}) and (\ref{definitionOfa}).
Hence $E_2^{p,2}=0$ unless $p=-3$, and $E_2^{-3,2}=\mathcal{O}(-1)^{\oplus a}$ by (\ref{Snpullback}).
It follows from (\ref{h1E-3haA}), 
(\ref{c2=5c3=3h1E(-2)}), (\ref{first vanishing}), and the assumption $h^1(\mathcal{E}(-1))=1$ (resp.\ $2$)
that the right $A$-module $\Ext^1(G,\mathcal{E})$ (resp.\ $\Ext^1(G,\mathcal{F})$) has the filtration which induces the following exact sequences
of right $A$-modules
\begin{gather*}
0\to F\to \Ext^1(G,\mathcal{E})\to S_3^{\oplus a}\to 0;\\
0\to S_1\to F\to S_2^{\oplus 2}\to 0.\\
\textrm{(resp.\ }0\to F\to \Ext^1(G,\mathcal{F})\to S_3^{\oplus a}\to 0;\\
0\to S_1^{\oplus 2}\to F\to S_2^{\oplus 2}\to 0.\textrm{)}
\end{gather*}
Correspondingly we obtain the following distinguished triangles
\begin{gather}
F\lotimes_AG\to \Ext^1(G,\mathcal{E})\lotimes_AG\to \mathcal{O}(-1)^{\oplus a}[3]\to ;\label{h0h1=01F}\\
\Omega_{\mathbb{P}^3}(1)[1]\to F\lotimes_AG\to \Omega_{\mathbb{P}^3}^2(2)^{\oplus 2}[2]\to \label{h0h1=01KeyDt}\\
\textrm{(resp.\ }F\lotimes_AG\to \Ext^1(G,\mathcal{F})\lotimes_AG\to \mathcal{O}(-1)^{\oplus a}[3]\to ;\label{h0h1=12F}\\
\Omega_{\mathbb{P}^3}(1)^{\oplus 2}[1]\to F\lotimes_AG\to \Omega_{\mathbb{P}^3}^2(2)^{\oplus 2}[2]\to \label{h0h1=12KeyDt}\textrm{)}
\end{gather}
by (\ref{Sjpullback}).
We see that $\mathcal{H}^{-1}(F\lotimes_AG)\cong E_2^{-1,1}$ by (\ref{h0h1=01F}) (resp. (\ref{h0h1=12F})).
Moreover we have $\mathcal{H}^{p}(F\lotimes_AG)=0$ unless $p=-2$ or $-1$
by (\ref{h0h1=01KeyDt}) (resp.(\ref{h0h1=12KeyDt})).
Set $\mathcal{M}=\mathcal{H}^{-2}(F\lotimes_AG)$. 
Then it follows from 
(\ref{h0h1=01F}) and (\ref{h0h1=01KeyDt}) 
(resp. (\ref{h0h1=12F}) and (\ref{h0h1=12KeyDt}))
that $\mathcal{M}$
fits in the following exact sequences
\begin{gather}
0\to E_2^{-3,1}\to \mathcal{O}(-1)^{\oplus a}\to \mathcal{M}\to E_2^{-2,1}\to 0;\label{E2-31MudaAri}\\
0\to \mathcal{M}\to \Omega_{\mathbb{P}^3}^2(2)^{\oplus 2}\xrightarrow{\varphi} \Omega_{\mathbb{P}^3}(1)\to E_2^{-1,1}\to 0.\label{h0h1=01KeyExSeq}\\
\textrm{(resp.\ }0\to \mathcal{M}\to \Omega_{\mathbb{P}^3}^2(2)^{\oplus 2}\xrightarrow{\varphi} \Omega_{\mathbb{P}^3}(1)^{\oplus 2}
\to E_2^{-1,1}\to 0.\textrm{)}\label{h0h1=12KeyExSeq}
\end{gather}
It follows from (\ref{h0E(0)andE(1)incasec2=5c3=3}) and 
the assumption $h^0(\mathcal{E}(-1))=0$ (resp.\ $1$) that the right $A$-module $\Hom(G,\mathcal{E})$ 
(resp.\ $\Hom(G,\mathcal{F})$)
is isomorphic to $S_0^{\oplus r+3}$ (resp.\ $S_0^{\oplus r-1}$).
Hence $E_2^{p,0}=0$ unless $p=0$, and 
$E_2^{0,0}$ is isomorphic to $\mathcal{O}_{\mathbb{P}^3}^{\oplus r+3}$ (resp.\ $\mathcal{O}_{\mathbb{P}^3}^{\oplus r-1}$) by (\ref{Sjpullback}).
Therefore $E_2^{p,q}=0$ unless $(p+q,q)=(0,0)$, $(0,1)$, $(-1,1)$, $(-1,2)$, or $(-2,1)$.
Then it follows from (\ref{BondalSpectralSequence}) that $E_2^{-3,1}=0$, and consequently the sequence (\ref{E2-31MudaAri})
becomes the following exact sequence
\begin{equation}
0\to \mathcal{O}(-1)^{\oplus a}\to \mathcal{M}\to E_2^{-2,1}\to 0.\label{E2-31MudaNashi}\\
\end{equation}
Besides (\ref{h0h1=01KeyExSeq}) (resp.\ (\ref{h0h1=12KeyExSeq})) and (\ref{E2-31MudaNashi}),
the spectral sequence (\ref{BondalSpectralSequence}) induces the following exact sequences
\begin{gather}
0\to E_2^{-2,1}\to \mathcal{O}_{\mathbb{P}^3}^{\oplus r+3}\to E_3^{0,0}\to 0;\label{surjOntoE300}\\
\textrm{(resp.\ }0\to E_2^{-2,1}\to \mathcal{O}_{\mathbb{P}^3}^{\oplus r-1}\to E_3^{0,0}\to 0;\textrm{)}\label{h0h1=12surjOntoE300}\\
0\to E_3^{-3,2}\to \mathcal{O}(-1)^{\oplus a}\to E_2^{-1,1}\to E_3^{-1,1}\to 0;\label{E3-11quotient}\\
0\to E_3^{-3,2}\to E_3^{0,0}\to E_4^{0,0}\to 0;\label{surjOntoE400}\\
0\to E_4^{0,0}\to \mathcal{E}\to E_3^{-1,1}\to 0.\label{h0h1=01EhasE3-11quotient}\\
\textrm{(resp.\ }0\to E_4^{0,0}\to \mathcal{F}\to E_3^{-1,1}\to 0.\textrm{)}\label{h0h1=12FhasE3-11quotient}
\end{gather}

For $i=1$ (resp.\ $i=1,2$) and $j=1,2$,
denote by $\varphi_{ij}$ the composite of the $j$-th inclusion $\Omega_{\mathbb{P}^3}^2(2)\hookrightarrow\Omega_{\mathbb{P}^3}^2(2)^{\oplus 2}$
and the morphism 
$\varphi$
appearing in the exact sequence (\ref{h0h1=01KeyExSeq}) (resp. (\ref{h0h1=12KeyExSeq})) and the $i$-th projection.

We claim here that 
$\varphi_{11}$ and $\varphi_{12}$
(resp. $\varphi_{i1}$ and $\varphi_{i2}$ for each $i$,
and $\varphi_{1j}$ and $\varphi_{2j}$ for each $j$)
are linearly independent
in $\Hom(\Omega_{\mathbb{P}^3}^2(2), \Omega_{\mathbb{P}^3}(1))$.
Indeed, if they are not, then
we may assume that 
$\varphi_{12}=0$
by some row or column elementary transformations;
thus $\varphi$ (resp.\ the composite of $\varphi$ and the first projection $\Omega_{\mathbb{P}^3}(1)^{\oplus 2}\to \Omega_{\mathbb{P}^3}(1)$)
factors through $\varphi_{11}$ via the first projection 
$\Omega_{\mathbb{P}^3}^2(2)^{\oplus 2}\to \Omega_{\mathbb{P}^3}^2(2)$.
Note here that the cokernel of $\varphi_{11}$ is isomorphic to the ideal sheaf $\mathcal{I}_p$ 
for some point $p\in \mathbb{P}^3$ by Lemma~\ref{BeilinsonGreat}
if $\varphi_{11}\neq 0$ and to $\Omega_{\mathbb{P}^3}(1)$ if $\varphi_{11}=0$.
Hence there exists a surjection $E_2^{-1,1}\to \mathcal{I}_p$
by (\ref{h0h1=01KeyExSeq})
(resp.\ (\ref{h0h1=12KeyExSeq}))
(even if $\varphi_{11}=0$).
Consider the composite of the morphism $\mathcal{O}(-1)^{\oplus a}\to E_2^{-1,1}$ in (\ref{E3-11quotient})
and the surjection $E_2^{-1,1}\to \mathcal{I}_p$.
Since the cokernel of the composite is a quotient of $E_3^{-1,1}$
and $a\leq 2$ by (\ref{aleq2}) and (\ref{definitionOfa}),
we see that 
$E_3^{-1,1}$ has $\mathcal{O}_L(-1)$ as a quotient for some line $L$ passing through $p$ in $\mathbb{P}^3$.
On the other hand, $E_3^{-1,1}$ can not have a negative degree sheaf as a quotient,
since it is a quotient of the nef vector bundle $\mathcal{E}$ by (\ref{h0h1=01EhasE3-11quotient})
(resp.\ (\ref{h0h1=12FhasE3-11quotient}) and (\ref{defofF})).
This is a contradiction. Therefore the claim holds.

In case $(h^0(\mathcal{E}(-1)), h^1(\mathcal{E}(-1)))=(1,2)$,
denote by $V$ the vector space generated by $\varphi_{ij}$ $(1\leq i,j\leq 2)$
in the following.
If $\dim V\leq 3$, then the claim above implies that
there exist $\lambda_{ij}\in K$ $(1\leq i, j\leq 2)$ such that 
$\lambda_{21}\varphi_{11}+\lambda_{22}\varphi_{21}=\lambda_{11}\varphi_{12}+\lambda_{12}\varphi_{22}\neq 0$.
Therefore, by changing the free basis consisting two elements of $\Omega_{\mathbb{P}^3}(1)^{\oplus 2}$,
we may assume that $\varphi_{21}=\varphi_{12}$ in case $\dim V\leq 3$.
Moreover if $\dim V=2$, we can make $\varphi_{12}$ to be zero 
after some row and column elementary transformations,
which contradicts the claim above.
Therefore $\dim V=3$ or $4$,
and we may furthermore assume that $\varphi_{11}$, $\varphi_{12}$, and $\varphi_{22}$ are linearly independent.

%Denote by $V$ the vector space generated by $\varphi_{11}$ and $\varphi_{12}$.
%(resp. Denote by $V$ the vector space generated by $\varphi_{ij}$ $(1\leq i,j\leq 2)$.
%If $\dim V\leq 3$, then the claim above implies that
%there exist $\lambda_{ij}\in K$ $(1\leq i, j\leq 2)$ such that 
%$\lambda_{21}\varphi_{11}+\lambda_{22}\varphi_{21}=\lambda_{11}\varphi_{12}+\lambda_{12}\varphi_{22}\neq 0$.
%Therefore, by changing the free basis consisting two elements of $\Omega_{\mathbb{P}^3}(1)^{\oplus 2}$,
%we may assume that $\varphi_{21}=\varphi_{12}$ in case $\dim V\leq 3$.
%Moreover if $\dim V=2$, we can make $\varphi_{12}$ to be zero 
%after some row and column elementary transformations,
%which contradicts the claim above.
%Therefore $\dim V=3$ or $4$,
%and we may furthermore assume that $\varphi_{11}$, $\varphi_{12}$, and $\varphi_{22}$ are linearly independent.)

Set $\varphi_0=\varphi_{11}$ and $\varphi_1=\varphi_{12}$.
(resp. Set $\varphi_0=\varphi_{11}$, $\varphi_1=\varphi_{12}$, $\varphi_2=\varphi_{21}$, and $\varphi_3=\varphi_{22}$.
Note that $\varphi_1=\varphi_2$ in case $\dim V=3$.)

Let $e_i^*$ be the image of $x_i^*$ via the isomorphism $H^0(\mathcal{O}(1))^*\to H^0(T_{\mathbb{P}^3}(-1))$.
Let $s$ be an element of $H^0(T_{\mathbb{P}^3}(-1))$,
and suppose that $s=\sum_{i=0}^3a_ie_i^*$ for some $a_i\in K$.
Note that $e_0^*=-\sum_{i=1}^3(x_i/x_0)e_i^*$ and $(e_1^*,e_2^*,e_3^*)$
is a local basis of $T_{\mathbb{P}^3}(-1)$ over 
$D_+(x_0)$, where $D_+(x_0)$ is the open set defined by $x_0\neq 0$.
Thus $s=\sum_{i=1}^3\{a_i-a_0(x_i/x_0)\}e_i^*$ over $D_+(x_0)$.
Let $(e_1,e_2,e_3)$ be the dual basis of $\Omega_{\mathbb{P}^3}(1)|_{D_+(x_0)}$.
Then the Koszul morphism $\varphi_s$ from $\Omega_{\mathbb{P}^3}^2(2)$ to $\Omega_{\mathbb{P}^3}(1)$
corresponding to the section $s$
has the representation matrix
\begin{equation}\label{UnivMatrix}
\begin{bmatrix}
0&a_3-a_0(x_3/x_0)&-(a_2-a_0(x_2/x_0))\\
-(a_3-a_0(x_3/x_0))&0&a_1-a_0(x_1/x_0)\\
a_2-a_0(x_2/x_0)&-(a_1-a_0(x_1/x_0))&0
\end{bmatrix}
\end{equation}
with respect to the bases $(e_2\wedge e_3,e_3\wedge e_1,e_1\wedge e_2)$ and $(e_1,e_2,e_3)$
over $D_+(x_0)$.
Denote this matrix by $M(a_0,a_1,a_2,a_3)$.

Now Lemma~\ref{BeilinsonGreat} enables us to assume,
by taking homogeneous coordinates $(x_0\colon x_1\colon x_2\colon x_3)$ of $\mathbb{P}^3$ suitably,
that $\varphi_i$ corresponds to 
the global section $e_{i}^*$ of $T_{\mathbb{P}^3}(-1)$
for $i=0,1$.
Similarly, in case $(h^0(\mathcal{E}(-1)), h^1(\mathcal{E}(-1)))=(1,2)$,
we may assume 
that $\varphi_i$ corresponds to 
the global section $e_{i}^*$ of $T_{\mathbb{P}^3}(-1)$ for $i=0,1,2,3$ if $\dim V=4$
and that $\varphi_i$ to $e_i^*$ except $i=2$
and $\varphi_2$ to $e_1^*$ if $\dim V=3$.

We claim here that $\coker([\varphi_0,\varphi_1]):\Omega_{\mathbb{P}^3}(1)\to \Coker([\varphi_0,\varphi_1])$
is the composite of the two restrictions
$\Omega_{\mathbb{P}^3}(1)\to \Omega_{\mathbb{P}^3}(1)|_L$ and $\Omega_{\mathbb{P}^3}(1)|_L\to \Omega_{L}(1)$
where $L$ is the line defined by $x_2=x_3=0$.
Note first that the morphism $[\varphi_0,\varphi_1]:\Omega_{\mathbb{P}^3}^2(2)^{\oplus 2}\to \Omega_{\mathbb{P}^3}(1)$
has the representation matrix
\begin{equation*}
\begin{bmatrix}
M(1,0,0,0)&M(0,1,0,0)
\end{bmatrix}
=
\begin{bmatrix}
0       &-x_3/x_0& x_2/x_0&0& 0&0\\
x_3/x_0 &       0&-x_1/x_0&0& 0&1\\
-x_2/x_0& x_1/x_0&       0&0&-1&0
\end{bmatrix}
\end{equation*}
with respect to the bases $(e_2\wedge e_3,e_3\wedge e_1,e_1\wedge e_2,e_2'\wedge e_3',e_3'\wedge e_1',e_1'\wedge e_2')$ and $(e_1,e_2,e_3)$
over $D_+(x_0)$,
where $(e_1',e_2',e_3')$ is the dual basis of the second component $\Omega_{\mathbb{P}^3}(1)|_{D_+(x_0)}$
of the direct sum $\Omega_{\mathbb{P}^3}^2(2)^{\oplus 2}$.
Since $e_i=x_0d(x_i/x_0)$ for $i=1,2,3$, 
the image $\im [\varphi_0,\varphi_1]|_{D_+(x_0)}$ of $[\varphi_0,\varphi_1]|_{D_+(x_0)}$ is
generated by $x_0d(x_2/x_0)$, $x_0d(x_3/x_0)$, $x_2d(x_1/x_0)$, and $x_3d(x_1/x_0)$.
Note that $x_2d(x_i/x_0)=(x_2/x_0)x_0d(x_i/x_0)$ and $x_3d(x_i/x_0)=(x_3/x_0)x_0d(x_i/x_0)$ $(1\leq i\leq 3)$
and that the quotient of $\Omega_{\mathbb{P}^3}(1)|_{D_+(x_0)}$ by the submodule generated by $x_2d(x_i/x_0)$ and $x_3d(x_i/x_0)$
$(1\leq i\leq 3)$
is $(\Omega_{\mathbb{P}^3}(1)\otimes \mathcal{O}_L)|_{D_+(x_0)}$,
where $L$ is the line defined by $x_2=x_3=0$.
Hence $\Coker([\varphi_0,\varphi_1])|_{D_+(x_0)}$ is nothing but the quotient $\Omega_{L}(1)|_{D_+(x_0)}$
of $(\Omega_{\mathbb{P}^3}(1)\otimes \mathcal{O}_L)|_{D_+(x_0)}$ by the submodule generated by 
$x_0d(x_2/x_0)\otimes 1$ and $x_0d(x_3/x_0)\otimes 1$.
Similarly $\Coker([\varphi_0,\varphi_1])|_{D_+(x_1)}$ is naturally isomorphic to $\Omega_{L}(1)|_{D_+(x_1)}$.
Moreover $\Coker([\varphi_0,\varphi_1])|_{D_+(x_i)}=0$ for $i=2$ and $3$.
Therefore the claim holds.

\subsubsection{The case where $(h^0(\mathcal{E}(-1)), h^1(\mathcal{E}(-1)))=(0,1)$}\label{(h^0({E}(-1)), h^1({E}(-1)))=(0,1)}
Since $\varphi=[\varphi_0,\varphi_1]$,
we see by the claim above 
that $\coker(\varphi):\Omega_{\mathbb{P}^3}(1)\to E_2^{-1,1}$
is nothing but the composite $\Omega_{\mathbb{P}^3}(1)\to \Omega_{\mathbb{P}^3}(1)|_L\to \Omega_{L}(1)$ of the restrictions,
where $L$ is a line.
Therefore $E_2^{-1,1}\cong \mathcal{O}_L(-1)$.
Then the sequence~(\ref{E3-11quotient}) implies that $E_3^{-1,1}$ is either 
$\mathcal{O}_L(-1)$ or zero,
but (\ref{h0h1=01EhasE3-11quotient}) and the nefness of $\mathcal{E}$ imply that the former can not happen.
Therefore $E_3^{-1,1}=0$ and $E_4^{0,0}\cong \mathcal{E}$;
thus we obtain a surjection $\mathcal{O}^{\oplus r+3}\to \mathcal{E}$
by (\ref{surjOntoE300}) and (\ref{surjOntoE400}).
In fact, this can not happen, as can be seen below.

Denote by $\mathcal{H}$ the kernel of the surjection $\mathcal{O}^{\oplus r+3}\to \mathcal{E}$;
$\mathcal{H}$ is a vector bundle of rank three.
Note here that the dual $\mathcal{H}^{\vee}$ of $\mathcal{H}$ is 
globally generated.
Moreover 
$c_1(\mathcal{H}^{\vee})=3$
and 
\[c_2(\mathcal{H}^{\vee})=c_2(\mathcal{H})=-c_1(\mathcal{H})c_1-c_2=4.\]
Furthermore 
\[
c_3(\mathcal{H}^{\vee})=-c_3(\mathcal{H})=c_2(\mathcal{H})c_1+c_1(\mathcal{H})c_2+c_3=0.
\]
Since $\mathcal{H}^{\vee}$ is a globally generated vector bundle of rank three with $c_3(\mathcal{H}^{\vee})=0$,
we have an exact sequence of vector bundles
\[0\to \mathcal{O}\to \mathcal{H}^{\vee}\to \mathcal{H}'\to 0.
\]
Then it follows from \cite[Proposition 2.4]{MR3019571} that $\mathcal{H}'\cong \mathcal{O}(1)\oplus\mathcal{O}(2)$.
This however contradicts that $c_2(\mathcal{H}^{\vee})=4$.
The case $(h^0(\mathcal{E}(-1)), h^1(\mathcal{E}(-1)))=(0,1)$
thus can not happen.

\subsubsection{The case where $(h^0(\mathcal{E}(-1)), h^1(\mathcal{E}(-1)))=(1,2)$}\label{(h^0({E}(-1)), h^1({E}(-1)))=(1,2)}
As we have shown
in \S\ref{SetUpFor(n,c_2, c_3)=(3,5,3)},
the composite $[\varphi_0,\varphi_1]$ of $\varphi$
and the first projection $\Omega_{\mathbb{P}^3}(1)^{\oplus 2}\to \Omega_{\mathbb{P}^3}(1)$
has
$\Omega_{L}(1)$ as its cokernel, 
where $L$ is the line defined by $x_2=x_3=0$.
Therefore $E_2^{-1,1}$ has $\mathcal{O}_L(-1)$ as a quotient
by (\ref{h0h1=12KeyExSeq}).

It follows from (\ref{UnivMatrix})
that $\varphi$ 
has,
with respect to the bases $(e_2\wedge e_3,e_3\wedge e_1,e_1\wedge e_2,e_2'\wedge e_3',e_3'\wedge e_1',e_1'\wedge e_2')$ and 
$(e_1,e_2,e_3,e_1',e_2',e_3')$
over $D_+(x_0)$,
the representation matrix
\begin{equation*}
\begin{bmatrix}
M(1,0,0,0)&M(0,1,0,0)\\
M(0,0,1,0)&M(0,0,0,1)
\end{bmatrix}
\end{equation*}
if $\dim V=4$,
and the representation matrix
\begin{equation*}
\begin{bmatrix}
M(1,0,0,0)&M(0,1,0,0)\\
M(0,1,0,0)&M(0,0,0,1)
\end{bmatrix}
\end{equation*}
if $\dim V=3$.
Therefore $\Coker(\varphi)$ is supported on the quadric surface
defined by $x_0x_3-x_1x_2=0$ if $\dim V=4$
and by $x_2^2=0$ if $\dim V=3$.
This implies that  $\varphi$ is injective.
Hence $\mathcal{M}=0$ by (\ref{h0h1=12KeyExSeq});
it follows from (\ref{E2-31MudaNashi}) that $a=0$ and $E_2^{-2,1}=0$;
thus $E_2^{-1,1}\cong E_3^{-1,1}$ by (\ref{E3-11quotient}).
Since $E_2^{-1,1}$ has $\mathcal{O}_L(-1)$ as a quotient,
so does $E_3^{-1,1}$.
However $E_3^{-1,1}$ can not have $\mathcal{O}_L(-1)$ 
as a quotient by (\ref{defofF}) and (\ref{h0h1=12FhasE3-11quotient})
since $\mathcal{E}$ is nef.
Therefore we conclude 
that the case $(h^0(\mathcal{E}(-1)), h^1(\mathcal{E}(-1)))=(1,2)$
does not happen.

\section{Proof of Theorem~\ref{c_1=3c_2<8}}\label{Proof of Main Theorem}
As we stated at the end of \S \ref{Introduction}, 
we always assume that $c_2=9$ in this section. 

\subsection{Proof for the case $n=2$ and $h^1(\mathcal{E})=0$}\label{Proof for the case n=2 and h^1(E)=0}
It follows from 
(\ref{h1E-2onP2}), 
(\ref{RRonP2(-1)}), and (\ref{RRonP2(0)})
that 
\begin{gather}
h^1(\mathcal{E}(-2))=6,\label{h1E-2onP2c2=9}\\
h^0(\mathcal{E}(-1))-h^1(\mathcal{E}(-1))=-3,\label{RRonP2(-1)c2=9}\\
h^0(\mathcal{E})=r.\label{RRonP2(0)c2=9}
\end{gather}
As in \S \ref{Set-up for the two-dimensional case}, set 
$e_{0,1}=h^0(\mathcal{E}(-1))$.
Then 
\begin{equation}\label{e01ha0ika}
e_{0,1}\leq 1.
\end{equation}
The proof of this fact 
goes almost identical to that 
in case $c_2=8$,
i.e., that in \cite[\S~3.1]{Nefofc1=3c2=8OnPN}: what we have to do is just to change several numerals
related to $c_2=8$ to those related to $c_2=9$; so we omit the proof of this fact.

We apply to $\mathcal{E}$ the Bondal spectral sequence 
(\ref{BondalSpectralSequence}).
As we have seen in 
\S\ref{Set-up for the two-dimensional case}
and Lemma~\ref{key} (2),
$E_2^{p,q}$ vanishes unless $(p,q)=(-2,1)$, $(-1,1)$ or $(0,0)$,
and $E_2^{-2,1}$ and $E_2^{-1,1}$ fit in one of the following exact sequences:
\begin{gather}
0\to E_2^{-2,1}\to \mathcal{O}(-3)\oplus \mathcal{O}(-1)
\xrightarrow{\nu_2}
\Omega_{\mathbb{P}^2}(1)^{\oplus e_{0,1}+1}
\to E_2^{-1,1}\to k(w)\to 0;\label{exactseqForc2=9}\\
0\to E_2^{-2,1}\to \mathcal{O}(-2)^{\oplus 2}
\xrightarrow{\mu_2}
\Omega_{\mathbb{P}^2}(1)^{\oplus e_{0,1}+1}
\to E_2^{-1,1}\to 0,\label{exactseqForc2=9WithoutBP}
\end{gather}
where $k(w)$ denotes the residue field of 
some point $w$ in $\mathbb{P}^2$.
Recall that each of the exact sequences 
is a consequence of the assumption $H^1(\mathcal{E})=0$.
Note that $E_2^{0,0}$ fits in the following exact sequence by (\ref{E_2^00 exact sequence in dim 2})
\begin{equation}\label{E_2^00 exact sequence in dim 2Forc2=9}
0\to \mathcal{O}^{\oplus 3e_{0,1}}\to 
\mathcal{O}(1)^{\oplus e_{0,1}}\oplus \mathcal{O}^{\oplus r}
\to E_2^{0,0}\to 0.
\end{equation}

\begin{lemma}\label{E2-11haCokerOftheEv}
If $E_2^{-2,1}=0$, then $E_2^{0,0}\cong E_3^{0,0}\cong \mathcal{O}(1)^{\oplus e_{0,1}}\oplus \mathcal{O}^{\oplus r-3e_{0,1}}$,
and $\mathcal{E}$ fits in the following exact sequence
\begin{equation}\label{(7.6)nobaainoE}
0\to \mathcal{O}(1)^{\oplus e_{0,1}}\oplus \mathcal{O}^{\oplus r-3e_{0,1}}\to \mathcal{E}\to E_2^{-1,1}\to 0.
\end{equation}
Moreover $H^i(E_2^{-1,1})=0$ for all $i$;
thus $E_2^{-1,1}$ is the cokernel of the evaluation map $H^0(\mathcal{E})\otimes\mathcal{O}\to \mathcal{E}$.
\end{lemma}
\begin{proof}
Since $E_2^{-2,1}=0$, $E_2^{0,0}\cong E_3^{0,0}$ by (\ref{E_3^{0,0}definition}),
and hence $E_2^{0,0}$ is torsion-free by (\ref{E_2^{-1,1}quotient}).
Therefore (\ref{E_2^00 exact sequence in dim 2Forc2=9}) shows that $\mathcal{O}(1)$ is a subsheaf of $E_2^{0,0}$
if $e_{0,1}=1$,
and consequently we see that $E_2^{0,0}$ is isomorphic to $\mathcal{O}(1)^{\oplus e_{0,1}}\oplus\mathcal{O}^{\oplus r-3e_{0,1}}$.
Hence it follows from (\ref{E_2^{-1,1}quotient})
that $\mathcal{E}$ fits in the desired exact sequence (\ref{(7.6)nobaainoE}).
Since we have (\ref{RRonP2(0)c2=9}),  
$H^0(\mathcal{O}(1)^{\oplus e_{0,1}}\oplus\mathcal{O}^{\oplus r-3e_{0,1}})$ is isomorphic to $H^0(\mathcal{E})$,
and thus
$H^0(E_2^{-1,1})$ vanishes.
Furthermore $H^i(E_2^{-1,1})=0$ for all $i$ since $h^1(\mathcal{E})=0$ by assumption and $h^2(\mathcal{E})=0$ by 
(\ref{H^2vanishing}).
\end{proof}

Let $\alpha_{ij}:\mathcal{O}(-2)\to \Omega_{\mathbb{P}^2}(1)$ be 
the composite of the $j$-th inclusion $\mathcal{O}(-2)\to \mathcal{O}(-2)^{\oplus 2}$,
$\mu_2$, and the $i$-th projection $\Omega_{\mathbb{P}^2}(1)^{\oplus e_{0,1}+1}\to \Omega_{\mathbb{P}^2}(1)$.
The $\alpha_{ij}$'s form a matrix $[\alpha_{ij}]$, which is another expression of $\mu_2$.

\begin{lemma}\label{property of mu_2}
For each $i$, $\alpha_{i1}$ and $\alpha_{i2}$ are linearly independent.
Similarly, for each $j$,  elements in the $j$-th column are linearly independent.
Furthermore the cokernel of the composite $[\alpha_{11},\alpha_{12}]$ of $\mu_2$ and the first projection 
$\Omega_{\mathbb{P}^2}(1)^{\oplus e_{0,1}+1}\to \Omega_{\mathbb{P}^2}(1)$
is supported on a cubic curve $E$, and $[\alpha_{11},\alpha_{12}]$ and $\mu_2$ are injective.
In particular $E_2^{-2,1}=0$ in (\ref{exactseqForc2=9WithoutBP}).
\end{lemma}
\begin{proof}
First note that $E_2^{-1,1}$ in (\ref{exactseqForc2=9WithoutBP}) has as a quotient
the cokernel of the composite, say $\psi$, of $\mu_2$ and any surjection 
$\Omega_{\mathbb{P}^2}(1)^{e_{0,1}+1}\to \Omega_{\mathbb{P}^2}(1)$.
Since $E_2^{-1,1}$ can not have $\Omega_{\mathbb{P}^2}(1)$ as a quotient  by Lemma~\ref{very often},
we observe that the composite $\psi$ can not be zero.

Suppose that
elements in one of the rows or the columns
are linearly dependent.
After some row or column elementary transformations, by replacing $\alpha_{ij}$,
we may then assume that $\alpha_{11}\neq 0$ and that $\alpha_{12}=0$
by the observation above.
The morphism $[\alpha_{11},0]$
factors through the first projection $\mathcal{O}(-2)^{\oplus 2}\to \mathcal{O}(-2)$,
and $\Coker([\alpha_{11},0])=\Coker(\alpha_{11})$.
If $\alpha_{11}$ factors through $\mathcal{O}(-1)$,
the cokernel of the induced morphism $\mathcal{O}(-1)\to\Omega_{\mathbb{P}^2}(1)$
is isomorphic to $\mathcal{I}_p$ for some point $p\in \mathbb{P}^2$
and $\Coker(\alpha_{11})$ fits in the following exact sequence
\[
0\to \mathcal{O}_L(-1)\to \Coker(\alpha_{11})\to \mathcal{I}_p\to 0,
\]
where $\mathcal{O}_L(-1)$ is the cokernel of the induced injection $\mathcal{O}(-2)\to \mathcal{O}(-1)$. 
In particular $\mathcal{O}_L(-1)$ is the torsion subsheaf of $\Coker(\alpha_{11})$,
and $\mathcal{I}_p$ is the quotient of $\Coker(\alpha_{11})$ by its torsion subsheaf.
This contradicts Lemma~\ref{very often}.
Therefore $\alpha_{11}$ does not factor through $\mathcal{O}(-1)$;
this implies that $\Coker(\alpha_{11})$ is isomorphic to $\mathcal{I}_Z(1)$,
where $Z$ is a $0$-dimensional closed subscheme of length three on $\mathbb{P}^2$.
This however can not happen by Lemma~\ref{very often}.
Therefore elements in the same row or column are linearly independent.

Since $\alpha_{11}$ and $\alpha_{12}$ are linearly independent,
the morphism $[\alpha_{11},\alpha_{12}]$ induces a non-zero morphism $\bar{\alpha}_{12}:\mathcal{O}(-2)\to \Coker(\alpha_{11})$,
and we see that $\Coker([\alpha_{11},\alpha_{12}])\cong \Coker(\bar{\alpha}_{12})$.
Denote by $\Coker(\alpha_{11})/\tors$ the quotient of $\Coker(\alpha_{11})$ by its torsion subsheaf.
As we have seen above, 
$\Coker(\alpha_{11})/\tors$
is either $\mathcal{I}_p$ or $\mathcal{I}_Z(1)$.
If $\Coker(\alpha_{11})/\tors=\mathcal{I}_p$, then $\Coker(\bar{\alpha}_{12})$ admits a negative degree line bundle on a conic
as a quotient,
which is impossible by Lemma~\ref{very often}.
Hence $\Coker(\alpha_{11})=\mathcal{I}_Z(1)$; thus $\Coker(\bar{\alpha}_{12})$ is supported on a cubic curve $E$ passing through $Z$.
Therefore $\bar{\alpha}_{12}$ is injective.
Since $\alpha_{11}$ is also injective,
so is $[\alpha_{11},\alpha_{12}]$.
Hence $\mu_2$ is injective, and thus $E_2^{-2,1}=0$.
\end{proof}

\begin{lemma}\label{(7.6)noBaaiNoKetsuron}
If $E_2^{-1,1}$ fits in the exact sequence (\ref{exactseqForc2=9WithoutBP}),
then it fits in the following exact sequence
\begin{equation}\label{(7.6)noShinkakei}
0\to \mathcal{O}(-2)^{\oplus e_{0,1}+3}
\to
\mathcal{O}(-1)^{\oplus 3e_{0,1}+3}
\to E_2^{-1,1}\to 0.
\end{equation}
\end{lemma}
\begin{proof}
By Lemma~\ref{property of mu_2}, $E_2^{-2,1}=0$,
and thus $E_2^{-1,1}$ fits in the following exact sequence
\begin{equation}\label{(18)noKagikaNoMoto}
0\to \mathcal{O}(-2)^{\oplus 2}
\xrightarrow{\mu_2}
\Omega_{\mathbb{P}^2}(1)^{\oplus e_{0,1}+1}
\to E_2^{-1,1}\to 0.
\end{equation}
Since $\Omega_{\mathbb{P}^2}(1)$ is isomorphic to $T_{\mathbb{P}^2}(-2)$,
the Euler sequence and the sequence above 
induce the desired exact sequence (\ref{(7.6)noShinkakei}).
\end{proof}

Denote by $\mathcal{K}$ the kernel of the surjection $E_2^{-1,1}\to k(w)$ in (\ref{exactseqForc2=9}).
\begin{lemma}\label{often}
The sheaf $\mathcal{K}$ can not admit the following sheaves as a quotient:
\begin{enumerate}
\item[$(1)$] $\Omega_{\mathbb{P}^2}(1)$;
\item[$(2)$] $\mathcal{I}_W(1)$, where $W$ is a $0$-dimensional closed subscheme of $\length W\geq 3$;
\item[$(3)$] $\mathcal{I}_W(2)$, where $W$ is a $0$-dimensional closed subscheme of $\length W\geq 6$.
\end{enumerate}
\end{lemma}
\begin{proof}
If $\mathcal{K}$ admits $\Omega_{\mathbb{P}^2}(1)$ as a quotient,
then $E_2^{-1,1}|_L$ admits $\mathcal{O}_L(-1)$ as a quotient for a line $L$ not passing through $w$.
This contradicts Lemma~\ref{very often}.
Suppose that $\mathcal{K}$ admits $\mathcal{I}_W(d)$ as a quotient, where $\length W\geq 3$ if $d=1$
and $\length W\geq 6$ if $d=2$.
Then we obtain the following exact sequence
\[0\to \mathcal{I}_W(d)\to \mathcal{D}\to k(w)\to 0\]
for some quotient sheaf $\mathcal{D}$ of $E_2^{-1,1}$.
Note that this sequence does not split since $E_2^{-1,1}$ does not admit $\mathcal{I}_W(d)$ as a quotient 
by Lemma~\ref{very often}.
Lemma~\ref{DhaO} then implies that $\mathcal{D}\cong \mathcal{I}_Z(d)$,
where $Z$ is a $0$-dimensional closed subscheme of $\length Z=\length W-1$.
However $E_2^{-1,1}$ can not admit $\mathcal{I}_Z(d)$ as a quotient by Lemma~\ref{very often}.
\end{proof}

Let $\beta_i:\mathcal{O}(-1)\to \Omega_{\mathbb{P}^2}(1)$ be the composite of the inclusion 
$\mathcal{O}(-1)\to \mathcal{O}(-3)\oplus\mathcal{O}(-1)$,
$\nu_2$, and the $i$-th projection $\Omega_{\mathbb{P}^2}(1)^{\oplus e_{0,1}+1}\to \Omega_{\mathbb{P}^2}(1)$,
and
let $\gamma_i:\mathcal{O}(-3)\to \Omega_{\mathbb{P}^2}(1)$ be the composite of the inclusion 
$\mathcal{O}(-3)\to \mathcal{O}(-3)\oplus\mathcal{O}(-1)$,
$\nu_2$, and the $i$-th projection $\Omega_{\mathbb{P}^2}(1)^{\oplus e_{0,1}+1}\to \Omega_{\mathbb{P}^2}(1)$.
We shall denote
by $\Coker(\gamma_i)/\tors$ the quotient of $\Coker(\gamma_i)$ by its torsion subsheaf.

\begin{lemma}\label{property of Coker(gammai)/tor}
If $\gamma_i\neq 0$, then 
$\Coker(\gamma_i)/\tors$ is isomorphic to $\mathcal{I}_W(d)$
for some $0$-dimensional closed subscheme $W$ of $\mathbb{P}^2$
and some integer $d$,
and the possible values of the pair $(\length W,d)$ are $(7,2)$, $(3,1)$, or $(1,0)$.
\end{lemma}
\begin{proof}
Suppose that $\gamma_i$ factors through $\mathcal{O}(-1)$.
The cokernel of the induced morphism $\mathcal{O}(-1)\to \Omega_{\mathbb{P}^2}(1)$ is
isomorphic to $\mathcal{I}_p$ for some point $p\in \mathbb{P}^2$,
and $\Coker(\gamma_i)$ fits in the following exact sequence
\[0\to \mathcal{O}_C(-1)\to \Coker(\gamma_i)\to \mathcal{I}_p\to 0,\]
where $C$ is a conic and $\mathcal{O}_C(-1)$ is the cokernel of the induced morphism $\mathcal{O}(-3)\to \mathcal{O}(-1)$.
Hence $(\length W,d)=(1,0)$ in this case.

Suppose that $\gamma_i$ does not factor through $\mathcal{O}(-1)$
and that it factors through $\mathcal{O}(-2)$.
Then the induced morphism $\mathcal{O}(-2)\to \Omega_{\mathbb{P}^2}(1)$ does not factor through $\mathcal{O}(-1)$
and
the cokernel of the morphism $\mathcal{O}(-2)\to \Omega_{\mathbb{P}^2}(1)$ is
isomorphic to $\mathcal{I}_Z(1)$ for some closed subscheme $Z\subset \mathbb{P}^2$ of length three.
Moreover $\Coker(\gamma_i)$ fits in the following exact sequence
\[0\to \mathcal{O}_L(-2)\to \Coker(\gamma_i)\to \mathcal{I}_Z(1)\to 0,\]
where $L$ is the line determined by the induced morphism $\mathcal{O}(-3)\to \mathcal{O}(-2)$
and $\mathcal{O}_L(-2)$ is the cokernel of that morphism.
Hence $(\length W,d)=(3,1)$ in this case.

Finally suppose that $\gamma_i$ does not factor through $\mathcal{O}(-1)$ nor $\mathcal{O}(-2)$.
Then $\gamma_i$ induces a regular section $s\in H^0(\Omega_{\mathbb{P}^2}(4))$, whose zero locus $(s)_0$ has length seven.
Hence $\Coker(\gamma_i)$ is isomorphic to $\mathcal{I}_W(2)$, where $W=(s)_0$.
Therefore $(\length W,d)=(7,2)$ in this case.
\end{proof}

\begin{lemma}\label{property of nu_2}
For any $i$, $\beta_{i}$ is non-zero.
\end{lemma}
\begin{proof}
Suppose, to the contrary, that $\beta_1=0$. 
Then the composite $[\gamma_1,\beta_1]$ of $\nu_2$ 
and the first projection $\Omega_{\mathbb{P}^2}(1)^{e_{0,1}+1}\to \Omega_{\mathbb{P}^2}(1)$
factors through $\mathcal{O}(-3)$ via the projection 
$\mathcal{O}(-3)\oplus \mathcal{O}(-1)\to \mathcal{O}(-3)$,
and $\Coker([\gamma_1,0])$ is isomorphic to $\Coker(\gamma_1)$.
Note that $\Coker(\gamma_1)$ is a quotient of $\mathcal{K}$.
It follows from Lemma~\ref{often} that $\gamma_1\neq 0$.
Then $\Coker(\gamma_1)/\tors\cong \mathcal{I}_W(d)$
for some $0$-dimensional closed subscheme $W$ of $\mathbb{P}^2$,
and the possible values of the pair $(\length W,d)$ are $(7,2)$, $(3,1)$, or $(1,0)$
by Lemma~\ref{property of Coker(gammai)/tor}.
However the values $(7,2)$ and $(3,1)$ are ruled out by Lemma~\ref{often}.
Hence $(\length W,d)=(1,0)$
and thus the torsion subsheaf of $\Coker(\gamma_1)$ is isomorphic to 
$\mathcal{O}_C(-1)$, where $C$ is a conic on $\mathbb{P}^2$.
Set $\{p\}=W$.
There exists the following commutative diagram with exact rows
\[
\xymatrix{
0\ar[r]&\mathcal{K}\ar[d]\ar[r]      &E_2^{-1,1}\ar[d]\ar[r]   &k(w)\ar[r]\ar@{=}[d]    & 0    \\
0\ar[r]&\Coker(\gamma_1) \ar[r]      &\mathcal{G}\ar[r]        &k(w)\ar[r]              & 0    
}
\]
for some quotient coherent sheaf $\mathcal{G}$ of $E_2^{-1,1}$.
Denote by $\mathcal{D}$ the cokernel of the composite of the two inclusions $\mathcal{O}_C(-1)\to \Coker(\gamma_1)$
and $\Coker(\gamma_1)\to \mathcal{G}$.
Then we have 
the following commutative diagram with exact rows.
\[
\xymatrix{
0\ar[r]&\Coker(\gamma_1)\ar[d]\ar[r]      &\mathcal{G}\ar[d]\ar[r]   &k(w)\ar[r]\ar@{=}[d]    & 0    \\
0\ar[r]&\mathcal{I}_p \ar[r]              &\mathcal{D}\ar[r]         &k(w)\ar[r]              & 0    
}
\]
Note here that the bottom row of the diagram above does not split
since $E_2^{-1,1}$
can not admit 
$\mathcal{I}_p$
as a quotient by Lemma~\ref{very often}.
Therefore Lemma~\ref{DhaO} implies 
that $w=p$
and 
that 
$\mathcal{D}$ is isomorphic to 
$\mathcal{O}_{\mathbb{P}^2}$.
Hence $\mathcal{G}$ fits in the following exact sequence
\[0\to \mathcal{O}_C(-1)\to \mathcal{G}\to \mathcal{O}_{\mathbb{P}^2}\to 0.\]
Let $\mathcal{E}_1$ be the kernel of the composite of the two surjections $\mathcal{E}\to \mathcal{G}$ and $\mathcal{G}\to \mathcal{O}_{\mathbb{P}^2}$.
Then $\mathcal{O}_C(-1)$ is a quotient of $\mathcal{E}_1$.
Hence $\mathcal{E}_1|_C$ admits $\mathcal{O}_C(-1)$ as a quotient, and it also fits in the following exact sequence
\[
0\to \mathcal{E}_1|_C\to \mathcal{E}|_C\to \mathcal{O}_C\to 0.
\]
Since $C$ is a conic, this implies that $\mathcal{E}|_C$ admits a negative degree quotient,
which contradicts Lemma~\ref{very often}.
\end{proof}
There exists the natural composition 
\[\Hom(\mathcal{O}(-1),\Omega_{\mathbb{P}^2}(1))\times \Hom(\mathcal{O}(-3),\mathcal{O}(-1))
\to \Hom(\mathcal{O}(-3),\Omega_{\mathbb{P}^2}(1)),\]
and we have two cases:
\begin{itemize}
\item $\gamma_1$ has $\beta_1$ as a factor;
\item $\gamma_1$ does not have $\beta_1$ as a factor.
\end{itemize}

\begin{lemma}\label{gamma1gabeta1deWarikirenai}
If $\gamma_1$ does not have $\beta_1$ as a factor,
then $E_2^{-2,1}=0$ and 
$E_2^{-1,1}$ fits in 
(\ref{(7.6)nobaainoE}) and (\ref{(7.6)noShinkakei}).
\end{lemma}
\begin{proof}
If $\gamma_1$ does not have $\beta_1$ as a factor,
then $[\gamma_1,\beta_1]$ induces a non-zero morphism $\bar{\gamma}_1:\mathcal{O}(-3)\to \Coker(\beta_1)$.
Since $\beta_1\neq 0$ by Lemma~\ref{property of nu_2},
this implies that $[\gamma_1,\beta_1]$ is injective.
Hence $\nu_2$ is injective and thus $E_2^{-2,1}=0$.
Lemma~\ref{E2-11haCokerOftheEv} then shows that $E_2^{-1,1}$ lies in (\ref{(7.6)nobaainoE}).
Since $E_2^{-2,1}=0$,
$\mathcal{K}$ fits in the following exact sequence
\[
0\to \mathcal{O}(-3)\oplus\mathcal{O}(-1)\to \Omega_{\mathbb{P}^2}(1)^{\oplus e_{0,1}+1}\to \mathcal{K}\to 0.  
\]
This together with
the isomorphism $\Omega_{\mathbb{P}^2}(1)\cong T_{\mathbb{P}^2}(-2)$
and the Euler sequence
induces the following exact sequence
\[
0\to \mathcal{O}(-3)\oplus\mathcal{O}(-2)^{\oplus e_{0,1}+1}\to \mathcal{O}(-1)^{\oplus 3e_{0,1}+2}\to \mathcal{K}\to 0.  
\]
Since we have an exact sequence
\begin{equation}\label{resolutionofkw}
0\to \mathcal{O}(-3)\to \mathcal{O}(-2)^{\oplus 2}\to \mathcal{O}(-1)\to k(w)\to 0,
\end{equation}
the resolution of $\mathcal{K}$ above implies that $E_2^{-1,1}$ fits in the following exact sequence
\begin{equation}\label{19doesnothappen}
0\to \mathcal{O}(-3)\xrightarrow{d_2} \mathcal{O}(-3)\oplus\mathcal{O}(-2)^{\oplus e_{0,1}+3}
\to \mathcal{O}(-1)^{\oplus 3e_{0,1}+3}\to E_2^{-1,1}\to 0.
\end{equation}
Lemma~\ref{E2-11haCokerOftheEv} and this sequence (\ref{19doesnothappen}) 
then 
show that 
$H^2(d_2)$
is an isomorphism;
we infer that the composite of $d_2$ and the projection 
$\mathcal{O}(-3)\oplus\mathcal{O}(-2)^{\oplus e_{0,1}+3}\to \mathcal{O}(-3)$ is an isomorphism.
Therefore the exact sequence (\ref{19doesnothappen}) is reduced to the exact sequence
(\ref{(7.6)noShinkakei}).
\end{proof}

Now if $E_2^{-1,1}$ lies in (\ref{exactseqForc2=9WithoutBP})
or if $E_2^{-1,1}$ lies in (\ref{exactseqForc2=9})
and $\gamma_1$ does not have $\beta_1$ as a factor,
then $E_2^{-1,1}$ lies also in (\ref{(7.6)nobaainoE}) and (\ref{(7.6)noShinkakei})
by Lemmas~\ref{E2-11haCokerOftheEv}, \ref{property of mu_2}, \ref{(7.6)noBaaiNoKetsuron}, and \ref{gamma1gabeta1deWarikirenai}.
Hence if $e_{0,1}=1$ then we obtain the following exact sequences:
\begin{gather}
0\to \mathcal{O}(-2)^{\oplus 4}\to \mathcal{O}(-1)^{\oplus 6}\to E_2^{-1,1}\to0;\label{E2-11noResolutionInCase18}\\
0\to \mathcal{O}(1)\oplus\mathcal{O}^{\oplus r-3}\to \mathcal{E}\to E_2^{-1,1}\to 0.\label{rank2E2-11}
\end{gather}
These two exact sequences imply that 
$\mathcal{E}$ fits in the exact sequence in case (18) of Theorem~\ref{c_1=3c_2<8}.
If $e_{0,1}=0$ then we obtain the following exact sequences:
\begin{gather}
0\to \mathcal{O}(-2)^{\oplus 3}\to \mathcal{O}(-1)^{\oplus 3}\to 
E_2^{-1,1}
\to 0;\label{O_E(D)resolutionbetter}\\
0\to \mathcal{O}^{\oplus r}\to \mathcal{E}\to 
E_2^{-1,1}
\to 0.\label{EandE(D)}
\end{gather}
These two exact sequences show that
$\mathcal{E}$ fits in the exact sequence in case (16) of Theorem~\ref{c_1=3c_2<8}.

\begin{question}
Does there exist a nef vector bundle fitting in the exact sequence in  case (18) of Theorem~\ref{c_1=3c_2<8} ?
\end{question}

\begin{rmk}\label{cokernelOf16and17}
If $E_2^{-1,1}$ is in 
(\ref{O_E(D)resolutionbetter}),
then $E_2^{-1,1}$ is 
supported on a cubic curve.
\end{rmk}

\begin{lemma}\label{gamma1gabeta1woInshiNiMotu}
If $\gamma_1$ has $\beta_1$ as a factor, then $e_{0,1}=0$,
$E_2^{-1,1}\cong \mathcal{O}$, and we have the following non-split exact sequences:
\begin{gather}
0\to \mathcal{O}(-3)\to \mathcal{O}^{\oplus r}\to E_3^{0,0}\to 0;\label{E_300EssentiallyTangent-1}\\
0\to E_{3}^{0,0}\to \mathcal{E}\to \mathcal{O}\to 0.\label{EhasOasaQuotient}
\end{gather}
In particular
$E_2^{-1,1}$ is the cokernel of the evaluation map $H^0(\mathcal{E})\otimes \mathcal{O}\to \mathcal{E}$.
\end{lemma}
\begin{proof}
Since $\gamma_1$ has $\beta_1$ as a factor,
by changing the free basis of $\mathcal{O}(-3)\oplus\mathcal{O}(-1)$,
we may assume that $\gamma_1=0$.

Suppose that $e_{0,1}=1$.
Then we have the following commutative diagram with exact rows.
\[
\xymatrix{
0\ar[r]&\mathcal{O}(-3)\ar_{\gamma_2}[d]\ar[r]&\mathcal{O}(-3)\oplus\mathcal{O}(-1)\ar[d]_{\nu_2}\ar[r]
&\mathcal{O}(-1)\ar[d]_{\beta_1}\ar[r]&0  \\
0\ar[r]&\Omega_{\mathbb{P}^2}(1)\ar[r]^{\tiny{\begin{bmatrix}0\\1\end{bmatrix}}}        &\Omega_{\mathbb{P}^2}(1)^{\oplus 2} \ar[r]^{[1,0]}       
&\Omega_{\mathbb{P}^2}(1)\ar[r]       &0 
}
\]
Since $\beta_1\neq 0$ by Lemma~\ref{property of nu_2},
$\beta_1$ is injective and $\Coker(\beta_1)=\mathcal{I}_p$ for some point $p\in \mathbb{P}^2$.
Hence 
the diagram above induces an exact sequence
\[
0\to \Coker(\gamma_2)\to \mathcal{K}\to \mathcal{I}_p\to 0.
\]
Denote by $\mathcal{D}$ the quotient of $E_2^{-1,1}$ by $\Coker(\gamma_2)$.
Then we 
obtain the following commutative diagram with exact rows.
\[
\xymatrix{
0\ar[r] &\mathcal{K}   \ar[r]   \ar[d]         &E_2^{-1,1} \ar[r] \ar[d] &k(w) \ar[r] \ar@{=}[d]   & 0    \\
0\ar[r] &\mathcal{I}_p \ar[r]                  &\mathcal{D}\ar[r]        &k(w) \ar[r]          & 0    \\
}
\]
Since 
$E_2^{-1,1}$ can not admit $\mathcal{I}_p$ as a quotient by Lemma~\ref{very often},
the bottom row of the diagram above does not split.
Hence 
$p=w$ and 
$\mathcal{D}$ is isomorphic to 
$\mathcal{O}_{\mathbb{P}^2}$
by Lemma~\ref{DhaO}.
The definition of $\mathcal{D}$ induces the following exact sequence
\[0\to \Coker(\gamma_2)\to E_2^{-1,1}\to \mathcal{O}\to 0.
\]
If $\gamma_2=0$, then $\Coker(\gamma_2)\cong\Omega_{\mathbb{P}^2}(1)$,
and thus $E_2^{-1,1}\cong \Omega_{\mathbb{P}^2}(1)\oplus \mathcal{O}$,
but this is absurd
by Lemma~\ref{very often}.
Therefore $\gamma_2\neq 0$,
and it follows from Lemma~\ref{property of Coker(gammai)/tor}
that $\Coker(\gamma_2)/\tors$ 
is isomorphic to $\mathcal{I}_W(d)$
for some $0$-dimensional closed subscheme $W$ of $\mathbb{P}^2$
and that the possible values of the pair $(\length W,d)$ are $(7,2)$, $(3,1)$, or $(1,0)$.
The sequence above induces
an
exact sequence
\[0\to \mathcal{I}_W(d)\to E_2^{-1,1}/\tors\to \mathcal{O}\to 0.\]
Now Lemma~\ref{conicpassing5points} implies that there exists a smooth rational curve $C$ such that 
the $d$-th twist $\mathcal{I}_W\cdot\mathcal{O}_C(d)$
of the inverse image ideal sheaf $\mathcal{I}_W\cdot\mathcal{O}_C$
is a negative degree line bundle.
This implies that $E_2^{-1,1}|_C$ admits a negative degree line bundle as a direct summand, which contradicts 
Lemma~\ref{very often}.
Therefore the case $e_{0,1}=1$ can not happen.

Suppose that $e_{0,1}=0$.
Since $\beta_1\neq 0$ by Lemma~\ref{property of nu_2}
and $\gamma_1=0$ by assumption,
the exact sequence~(\ref{exactseqForc2=9}) shows that $E_2^{-2,1}\cong \mathcal{O}(-3)$ and that $\mathcal{K}\cong \Coker(\beta_1)\cong \mathcal{I}_p$.
Since the sequence
\[0\to \mathcal{I}_p\to E_2^{-1,1}\to k(w)\to 0
\]
does not split by Lemma~\ref{very often},
we see that
$p=w$ and that 
\[E_2^{-1,1}\cong \mathcal{O}_{\mathbb{P}^2}\]
by Lemma~\ref{DhaO}.
Note here that 
$E_2^{0,0}\cong \mathcal{O}^{\oplus r}$
by (\ref{E_2^00 exact sequence in dim 2Forc2=9}).
Hence we obtain the desired exact sequences 
(\ref{E_300EssentiallyTangent-1}) and (\ref{EhasOasaQuotient})
by 
(\ref{E_3^{0,0}definition}) and (\ref{E_2^{-1,1}quotient}).
Finally 
the sequence (\ref{EhasOasaQuotient}) does not split
since $H^0(E_3^{0,0})\cong H^0(\mathcal{E})$ by (\ref{RRonP2(0)c2=9}) and (\ref{E_300EssentiallyTangent-1}).
\end{proof}

Suppose that $\gamma_1$ has $\beta_1$ as a factor.
We shall give a resolution of $\mathcal{E}$ in terms of a full strong exceptional collection
$(\mathcal{O}(-3),\mathcal{O}(-2),\mathcal{O}(-1))$ of line bundles on $\mathbb{P}^2$.
By composing the Euler exact sequence and the dual of that,
we obtain the following exact sequence
\begin{equation}\label{ResolutionOfO}
0\to \mathcal{O}(-3)\to \mathcal{O}(-2)^{\oplus 3}\to \mathcal{O}(-1)^{\oplus 3}\to \mathcal{O}\to 0.
\end{equation}
Hence we obtain the following exact sequence
\[0\to \mathcal{O}(-3)^{\oplus r}\to \mathcal{O}(-2)^{\oplus 3r}\to \mathcal{O}(-1)^{\oplus 3r}\to \mathcal{O}^{\oplus r}\to 0.\]
The exact sequence (\ref{E_300EssentiallyTangent-1}) then induces the following exact sequence
\begin{equation}\label{longResolutionOfE300}
0\to \mathcal{O}(-3)^{\oplus r}\to \mathcal{O}(-2)^{\oplus 3r}\oplus \mathcal{O}(-3)\to \mathcal{O}(-1)^{\oplus 3r}\to E_3^{0,0}\to 0.
\end{equation}
The sequence (\ref{EhasOasaQuotient}) together with (\ref{longResolutionOfE300}) and (\ref{ResolutionOfO})
induces the following
resolution of $\mathcal{E}$:
\[0\to \mathcal{O}(-3)^{\oplus r+1}\xrightarrow{d_2} \mathcal{O}(-2)^{\oplus 3r+3}\oplus \mathcal{O}(-3)
\to \mathcal{O}(-1)^{\oplus 3r+3}\to \mathcal{E}\to 0.\]
Since $h^1(\mathcal{E})=0$ by assumption, the linear map $H^2(d_2)$ is surjective.
Hence the composite of $d_2$ and the projection $\mathcal{O}(-2)^{\oplus 3r+3}\oplus \mathcal{O}(-3)\to \mathcal{O}(-3)$
is non-zero and thus a surjection.
Therefore the resolution above is reduced to 
that in case (17) of Theorem~\ref{c_1=3c_2<8}.
This completes the proof of Theorem~\ref{c_1=3c_2<8} for the case $n=2$, $c_2=9$, and $h^0(\mathcal{E})=0$.

Here we give the following two remarks about the property of $E_2^{-1,1}$ in 
(\ref{E2-11noResolutionInCase18}),
but they do not effect the rest of the proof at all,
so the reader may skip them.
\begin{rmk}
Since $\mathcal{O}(1)^{\oplus 6}$ is globally generated of rank six, 
a general morphism from 
$\mathcal{O}(-2)^{\oplus 4}$ to $\mathcal{O}(-1)^{\oplus 6}$
is a subbundle morphism, and hence its cokernel is a vector bundle.
On the other hand, $E_2^{-1,1}$
is not a vector bundle;
indeed, 
since $E_2^{-1,1}$ lies also in (\ref{rank2E2-11}),
we see $c_1(E_2^{-1,1})=2$ and $c_2(E_2^{-1,1})=7$,
and if $E_2^{-1,1}$ were a vector bundle, it would be a nef vector bundle, so that $c_2(E_2^{-1,1})\leq c_1(E_2^{-1,1})^2=4$,
which is absurd.
Therefore the morphism $\mathcal{O}(-2)^{\oplus 4}\to \mathcal{O}(-1)^{\oplus 6}$ in (\ref{E2-11noResolutionInCase18})
is not a general morphism.
\end{rmk}
Recall that (\ref{E2-11noResolutionInCase18}) comes from the exact sequence (\ref{(18)noKagikaNoMoto})
with $e_{0,1}=1$:
\[
0\to \mathcal{O}(-2)^{\oplus 2}
\xrightarrow{\mu_2}
\Omega_{\mathbb{P}^2}(1)^{\oplus 2}
\to E_2^{-1,1}\to 0.
\]
\begin{rmk}
We have the following commutative diagram with exact rows
\[
\xymatrix{
0\ar[r] &\mathcal{O}(-2)^{\oplus 2}\ar[r]^{\mu_2}\ar@{=}[d]                   &\Omega_{\mathbb{P}^2}(1)^{\oplus 2}\ar[r]\ar[d]_{[1,0]}         &E_2^{-1,1}\ar[r]\ar[d]         & 0    \\
0\ar[r] &\mathcal{O}(-2)^{\oplus 2}\ar[r]^{[\alpha_{11},\alpha_{12}]} &\Omega_{\mathbb{P}^2}(1)\ar[r]                          &\Coker([\alpha_{11},\alpha_{12}])\ar[r]  & 0    
}
\]
and an exact sequence
\[
0\to \Omega_{\mathbb{P}^2}(1)\to E_2^{-1,1}\to \Coker([\alpha_{11},\alpha_{12}])\to 0.
\]
Since $\alpha_{11}$ and $\alpha_{12}$ are linearly independent
and so are $\alpha_{11}$ and $\alpha_{21}$,
there exist no elements $\lambda \in K$ and $\lambda_{ij}\in K$ $(1\leq i,j\leq 2)$ such that 
\[
\begin{bmatrix}
\alpha_{11}&\alpha_{12}\\
\alpha_{21}&\alpha_{22}
\end{bmatrix}
\begin{bmatrix}
\lambda_{11}&\lambda_{12}\\
\lambda_{21}&\lambda_{22}
\end{bmatrix}
=
\begin{bmatrix}
1\\
\lambda
\end{bmatrix}
\begin{bmatrix}
\alpha_{11}&\alpha_{12}
\end{bmatrix}.
\]
Hence no splitting injection ${}^t[1,\lambda]$ ($\lambda\in K$) of the projection $[1,0]$
induces a splitting injection of 
the surjection $E_2^{-1,1}\to \Coker([\alpha_{11},\alpha_{12}])$.
\end{rmk}

\subsection{The sheaf $E_2^{-1,1}$ in (\ref{rank2E2-11}) is torsion-free}\label{proofoftorsion-free}
We shall apply 
the following lemma in \S\ref{Proof for the case n=2 and h^1(E)>0}.
\begin{lemma}\label{torsion-free}
The sheaf $E_2^{-1,1}$ in (\ref{rank2E2-11}) is a torsion-free sheaf of rank two
with $c_1(E_2^{-1,1})=2$ and $c_2(E_2^{-1,1})=7$.
\end{lemma}
\begin{proof}
Denote by $\mathcal{F}$ the cokernel of the composite of the two inclusions
$\mathcal{O}(1)\to \mathcal{O}(1)\oplus\mathcal{O}^{\oplus r-3}$
and $\mathcal{O}(1)\oplus\mathcal{O}^{\oplus r-3}\to \mathcal{E}$.
Then $c_1(\mathcal{F})=2$ and $c_2(\mathcal{F})=7$.
Moreover $\mathcal{F}$ fits in the following exact sequence
\[0\to \mathcal{O}^{\oplus r-3}\to \mathcal{F}\to E_2^{-1,1}\to 0.\]
Let $(e_1,e_2,\dots, e_{r-3})$ be the free basis of $\mathcal{O}^{\oplus r-3}$.
For $0\leq s\leq r-3$, denote by $\mathcal{O}^{\oplus s}$ the submodule generated by the set
$\{e_i|i\leq s\}$ in $\mathcal{O}^{\oplus r-3}$
and by $\mathcal{F}_s$ the quotient of $\mathcal{F}$ by $\mathcal{O}^{\oplus s}$.
We have the following exact sequence
\[0\to \mathcal{O}\to\mathcal{F}_s\to \mathcal{F}_{s+1}\to 0.\]
Note that $\mathcal{F}_0=\mathcal{F}$ and that $\mathcal{F}_{r-3}=E_2^{-1,1}$.

We shall show that $\mathcal{F}_s$ is torsion-free for all $s$ $(0\leq s\leq r-3)$
by induction on $s$.
Now consider the case $s=0$.
Since we have (\ref{first assumption}) and $\mathcal{E}$ is locally free, 
the singular locus of $\mathcal{F}$ has codimension $\geq 2$.
Hence it follows from \cite[Lemma 5.4]{resolution} that $\mathcal{F}$ is torsion-free;
thus the claim holds for the case $s=0$.

Suppose that $\mathcal{F}_s$ is torsion-free. We shall show that $\mathcal{F}_{s+1}$ is torsion-free
by contradiction.
For simplicity, by changing the symbols, we denote $\mathcal{F}_s$ by $\mathcal{F}$
and $\mathcal{F}_{s+1}$ by $\mathcal{F}_+$,
and assume that $\mathcal{F}$ is a torsoin-free sheaf with $c_1(\mathcal{F})=2$ and $c_2(\mathcal{F})=7$,
that $\mathcal{F}$ can not admit a negative degree line bundle on a line or a conic as a quotient,
and that $\mathcal{F}_+$ is not torsion-free.
Let $\mathcal{F}^{\vee\vee}$ be the double dual of $\mathcal{F}$.
Then $\mathcal{F}^{\vee\vee}$ is a nef vector bundle with first Chern class two 
by Lemma~\ref{doubledualoftorsionfreequotientofnefvbonSmoothsurfaceisanefvb}.
Denote by $\mathcal{Q}$ the cokernel of the inclusion $\mathcal{F}\to \mathcal{F}^{\vee\vee}$.
Then we have 
\begin{equation}\label{lengthQ}
\length \mathcal{Q}=c_2(\mathcal{F})-c_2(\mathcal{F}^{\vee\vee})=7-c_2(\mathcal{F}^{\vee\vee}).
\end{equation}

Denote by $\alpha$ the composite of the two inclusions $\mathcal{O}\to \mathcal{F}$
and $\mathcal{F}\to \mathcal{F}^{\vee\vee}$,
and 
let $\mathcal{F}^{\vee\vee}/\mathcal{O}$ be the cokernel of $\alpha$.
If $h^0(\mathcal{F}^{\vee\vee}(-1))=0$, then $\mathcal{F}^{\vee\vee}/\mathcal{O}$ is torsion-free by \cite[Lemma 5.4]{resolution}.
Since $\mathcal{F}_{+1}$ is a subsheaf of $\mathcal{F}^{\vee\vee}/\mathcal{O}$
and it has  a non-zero torsion subsheaf, we infer that 
$h^0(\mathcal{F}^{\vee\vee}(-1))\neq 0$.
It then follows from \cite[Theorem 6.5]{resolution}
that $\mathcal{F}^{\vee\vee}$ satisfies one of the following:
\begin{enumerate}
\item[(1)] $\mathcal{F}^{\vee\vee}\cong \mathcal{O}(2)\oplus\mathcal{O}^{\oplus f-1}$;
\item[(2)] $\mathcal{F}^{\vee\vee}\cong \mathcal{O}(1)^{\oplus 2}\oplus\mathcal{O}^{\oplus f-2}$;
\item[(3)] $\mathcal{F}^{\vee\vee}$ fits in an exact sequence
\[0\to\mathcal{O}(-1)\to \mathcal{O}(1)\oplus\mathcal{O}^{\oplus f}\to \mathcal{F}^{\vee\vee}\to 0.\]
\end{enumerate}
In particular, we see that $c_2(\mathcal{F}^{\vee\vee})\leq 2$. Therefore $\length \mathcal{Q}\geq 5$
by (\ref{lengthQ}).
Since $\mathcal{F}$ does not admit a negative degree line bundle on a line as a quotient,
we infer that $\length \mathcal{Q}|_L\leq 2$ for any line $L$.
Similarly we infer that $\length \mathcal{Q}|_C\leq 4$ for any smooth conic $C$
by the assumption on $\mathcal{F}$.
If $\mathcal{Q}$ is generated by a single element, then $\mathcal{Q}$ is isomorphic to 
the structure sheaf $\mathcal{O}_Z$ of 
the scheme-theoretic support $Z$ of $\mathcal{Q}$.
Lemma~\ref{conicpassing5points} then shows the existence of a conic $C$ 
such that $\length Z\cap C\geq 5$.
Hence $\length \mathcal{Q}|_C=\length \mathcal{O}_Z|_C\geq 5$, which is a contradiction.
In the following, we assume that $\mathcal{Q}$ is not generated by a single element.

Suppose that the case (1) holds. 
Let $\mathcal{Q}_1$ be the image of the composite of the inclusion $\mathcal{O}(2)\to \mathcal{F}^{\vee\vee}$
and the surjection $\mathcal{F}^{\vee\vee}\to \mathcal{Q}$.
Denote by $\mathcal{Q}_2$ the quotient of $\mathcal{Q}$ by $\mathcal{Q}_1$.
Then we have a surjection $\mathcal{O}^{\oplus f-1}\to \mathcal{Q}_2$.
Let $\mathcal{G}$ be the kernel of this surjection $\mathcal{O}^{\oplus f-1}\to \mathcal{Q}_2$.
Then we have a surjection $\mathcal{F}\to \mathcal{G}$.
Suppose that $\mathcal{Q}_2\neq 0$.
Then 
there exists a line $L$ such that 
$\length \mathcal{Q}_2|_L\geq 1$.
This implies that 
the kernel of the surjection $\mathcal{O}_L^{\oplus f-1}\to \mathcal{Q}_2|_L$,
i.e., a quotient of 
$\mathcal{G}|_L$ 
has negative degree.
Since we have a surjection 
$\mathcal{F}|_L\to \mathcal{G}|_L$,
this contradicts the assumption on $\mathcal{F}$. 
Therefore $\mathcal{Q}_2=0$; thus 
$\mathcal{Q}_1\cong \mathcal{Q}$.
Hence we obtain a surjection $\mathcal{O}(2)\to \mathcal{Q}\cong\mathcal{Q}(2)$,
and thus $\mathcal{Q}$ is generated by a single element.
This contradicts the assumption.
Hence the case (1) does not happen.

Suppose that 
we are in case (2) or (3).
We have an inclusion $\mathcal{O}(1)\to\mathcal{F}^{\vee\vee}$.
Denote by $\mathcal{F}'$ 
the cokernel of the inclusion.
If $\mathcal{F}^{\vee\vee}$ lies in the case (2),
then $\mathcal{F}'\cong \mathcal{O}(1)\oplus \mathcal{O}^{\oplus f-2}$.
If $\mathcal{F}^{\vee\vee}$ lies in the case (3),
then $\mathcal{F}'$ is a torsion-free sheaf with $c_1(\mathcal{F}')=1$
and $c_2(\mathcal{F}')=1$. 
Hence its double dual $\mathcal{F}'^{\vee\vee}$ is either $\mathcal{O}(1)\oplus \mathcal{O}^{\oplus f-2}$
or $T_{\mathbb{P}^2}(-1)\oplus \mathcal{O}^{\oplus f-3}$.
If $\mathcal{F}'^{\vee\vee}\cong T_{\mathbb{P}^2}(-1)\oplus \mathcal{O}^{\oplus f-3}$, 
then $c_2(\mathcal{F}')=1=c_2(\mathcal{F}'^{\vee\vee})$;
thus $\mathcal{F}'$ and $\mathcal{F}'^{\vee\vee}$ are isomorphic. In particular $\mathcal{F}'$ is locally free.
If $\mathcal{F}'^{\vee\vee}\cong \mathcal{O}(1)\oplus \mathcal{O}^{\oplus f-2}$, then
$\mathcal{F}'$ is 
$\mathcal{I}_q(1)\oplus \mathcal{O}^{\oplus f-2}$ for some point $q$.
Let 
$\mathcal{H}_1$ and 
$\mathcal{Q}_1$ be 
respectively the kernel and 
the image 
of the composite of the inclusion $\mathcal{O}(1)\to \mathcal{F}^{\vee\vee}$
and the projection $\mathcal{F}^{\vee\vee}\to \mathcal{Q}$.
Denote by $\mathcal{G}$ the quotient of $\mathcal{F}$ by $\mathcal{H}_1$
and by $\mathcal{Q}_2$ the quotient of $\mathcal{Q}$ by $\mathcal{Q}_1$.
Note that $\mathcal{H}_1\cong \mathcal{I}_{Z_1}(1)$ for some $0$-dimensional closed subscheme $Z_1$
and that $\mathcal{Q}_1\cong \mathcal{O}_{Z_1}$.
Then we have the following commutative diagram with exact rows and columns.
\[
\xymatrix{
       &0\ar[d]                         &0\ar[d]                           &0\ar[d]                  &     \\
0\ar[r]&\mathcal{I}_{Z_1}(1)\ar[r]\ar[d]&\mathcal{O}(1)\ar[r]\ar[d]        &\mathcal{O}_{Z_1}\ar[r]\ar[d]&0    \\
0\ar[r]&\mathcal{F}         \ar[r]\ar[d]&\mathcal{F}^{\vee\vee}\ar[r]\ar[d]&\mathcal{Q}\ar[r]\ar[d]  &0    \\
0\ar[r]&\mathcal{G}         \ar[d]\ar[r]&\mathcal{F}'\ar[d]\ar[r]          &\mathcal{Q}_2\ar[r]\ar[d]&0     \\
       &0                               &0                                 &0                        &      \\
}
\]

We claim that $\length \mathcal{Q}_2|_L\leq 1$ for any line $L$; this follows
from the fact that $\mathcal{G}|_L$ can not admit a line bundle of negative degree
as a quotient,
but a careful argument is needed if $\mathcal{F}'\cong \mathcal{I}_q(1)\oplus \mathcal{O}^{\oplus f-1}$
and $q\in L$;
suppose that $\length \mathcal{Q}_2|_L\geq 1$ for a line $L$ passing through $q$. 
Note here that $\mathcal{I}_q(1)|_L$ is isomorphic to $\mathcal{O}_L\oplus k(q)$.
Since $\mathcal{G}|_L$ does not admit a line bundle of negative degree as a quotient,
we infer that the composite of the inclusion $k(q)\to \mathcal{I}_q(1)|_L\to \mathcal{F}'|_L$
and the surjection $\mathcal{F}'|_L\to \mathcal{Q}_2|_L$ must be surjective.
Hence $\length \mathcal{Q}_2|_L=1$.  Therefore the claim holds,

The claim above implies, in particular, that the number of minimal (non-zero) generators of $\mathcal{Q}_{2,p}$ at each point $p$
is at most one.
Hence it follows from the claim above that $\mathcal{Q}_2$ is isomorphic to the residue field $k(p)$ of some point $p$
unless $\mathcal{Q}_2$ is zero.

If $\mathcal{Q}_2=0$, then 
$\mathcal{Q}\cong \mathcal{Q}_1\cong \mathcal{O}_{Z_1}$,
and this case does not happen
by our assumption.
Suppose that $\mathcal{Q}_2\cong k(p)$.
Then $\length Z_1\geq 4$.
If $p\notin Z_1$, then the support $Z$ of $\mathcal{Q}$ is the disjoint union of $Z_1$ and $p$,
and we see that $\mathcal{Q}\cong \mathcal{O}_Z$; this case is also ruled out.
In the following, we assume that $p\in Z_1$.

We claim that $\length Z_1\cap L\leq 2$ for any line $L$.
If $\mathcal{O}_{Z_1}|_L\to \mathcal{Q}|_L$ is injective,
then the claim holds since $\length \mathcal{Q}|_L\leq 2$.
Suppose that $\mathcal{O}_{Z_1}|_L\to \mathcal{Q}|_L$ is not injective.
Then $p\in L$.
Since $\mathcal{O}_L(1)\to \mathcal{F}^{\vee\vee}|_L$ is injective, we have the following exact sequence
\[
\kappa\to \mathcal{O}_{Z_1}|_L\to \mathcal{Q}|_L\to k(p)\to 0,
\]
where $\kappa$ denotes the kernel of the morphism $\mathcal{G}|_L\to \mathcal{F}'|_L$.
Since we have the following commutative diagram with exact rows,
\[
\xymatrix{
0\ar[r]&\mathcal{G}(-1)\ar[r]\ar[d]    &\mathcal{G}\ar[r]\ar[d]   &\mathcal{G}|_L\ar[r]\ar[d]&0    \\
0\ar[r]&\mathcal{F}'(-1)\ar[r]&\mathcal{F}'\ar[r]&\mathcal{F}'|_L\ar[r]             &0     \\
}
\]
we see that 
$\kappa\cong k(p)$.
Since $\mathcal{O}_{Z_1}|_L\to \mathcal{Q}|_L$ is not injective,
this implies that there exists the following exact sequence 
\[
0\to k(p)\to \mathcal{O}_{Z_1}|_L\to \mathcal{Q}|_L\to k(p)\to 0.
\]
Therefore $\length Z_1\cap L=\length \mathcal{O}_{Z_1}|_L=\length \mathcal{Q}|_L\leq 2$.

We claim here that the torsion subsheaf of  $\mathcal{F}^{\vee\vee}/\mathcal{O}$ is isomorphic to $\mathcal{O}_L(1)$
for some line $L$ containing $p$.
First note that if $\alpha$ factors through some subsheaf $\mathcal{O}(1)$
then the claim holds.
Suppose that $\mathcal{F}^{\vee\vee}\cong \mathcal{O}(1)^{\oplus 2}\oplus \mathcal{O}^{\oplus f-2}$.
If $\alpha$
does not factor through $\mathcal{O}(1)^{\oplus 2}$,
then $\mathcal{F}^{\vee\vee}/\mathcal{O}$ becomes locally free. This contradicts that $\mathcal{F}_+$ has a non-zero tosion subsheaf.
Therefore $\alpha$
factors through $\mathcal{O}(1)^{\oplus 2}$.
Since $\mathcal{F}^{\vee\vee}/\mathcal{O}$ can not be torsion-free, we see that the induced morphism 
$\mathcal{O}\to \mathcal{O}(1)^{\oplus 2}$ factors through a direct summand $\mathcal{O}(1)$.
Therefore the claim holds in this case.
Suppose that we are in case (3).
Then $\mathcal{F}'$ is either $T_{\mathbb{P}^2}(-1)\oplus \mathcal{O}^{\oplus f-3}$
or $\mathcal{I}_p(1)\oplus \mathcal{O}^{\oplus f-2}$.
If $\alpha$
does not factor through $\mathcal{O}(1)$,
then we have an injection $\mathcal{O}\to \mathcal{F}'$.
If $\mathcal{F}'=T_{\mathbb{P}^2}(-1)\oplus \mathcal{O}^{\oplus f-3}$,
then it follows from \cite[Lemma 5.4]{resolution} that 
$\mathcal{F}'/\mathcal{O}$ is torsion-free; thus $\mathcal{F}^{\vee\vee}/\mathcal{O}$ is torsion-free.
This is a contradiction.
If $\mathcal{F}'=\mathcal{I}_p(1)\oplus \mathcal{O}^{\oplus f-2}$,
the cokernel of the morphism $\mathcal{O}\to \mathcal{F}'$ is isomorphic to $\mathcal{O}_L\oplus \mathcal{O}^{\oplus f-2}$
for some line $L$ containing $p$, since $\mathcal{F}^{\vee\vee}/\mathcal{O}$ is not torsion-free.
Now we have the following commutative diagram with exact rows.
\[
\xymatrix{
0\ar[r]&\mathcal{F}_+\ar[r]\ar[d]    &\mathcal{F}^{\vee\vee}/\mathcal{O}\ar[r]\ar[d]   &\mathcal{Q}\ar[r]\ar[d]&0    \\
0\ar[r]&\mathcal{G}/\mathcal{O}\ar[r]&\mathcal{O}_L\oplus\mathcal{O}^{\oplus f-2}\ar[r]&k(p)\ar[r]             &0     \\
}
\]
We see that $(\mathcal{G}/\mathcal{O})|_L$ admits a negative degree line bundle as a quotient. This is a contradiction.
Therefore $\alpha$ factors through the subsheaf $\mathcal{O}(1)$ 
and the claim also holds in this case.

The claim above implies that 
we have the following commutative diagram with exact rows and columns
\[
\xymatrix{
0\ar[r]&\mathcal{I}_{Z_1}\cdot\mathcal{O}_L(1)\ar[r]\ar[d]&\mathcal{O}_L(1)\ar[r]\ar[d]            &\mathcal{O}_{Z_1}\ar[r]\ar[d]&0    \\
0\ar[r]&\mathcal{F}_+\ar[r]                               &\mathcal{F}^{\vee\vee}/\mathcal{O}\ar[r]&\mathcal{Q}\ar[r]            &0    
}
\]
where $\mathcal{I}_{Z_1}\cdot\mathcal{O}_L$ denotes the inverse image ideal sheaf.
In particular, we infer that $Z_1$ lies on the line $L$. Since $\length Z_1\geq 4$, this contradicts
the claim that $\length Z_1\cap L\leq 2$.
Therefore we conclude that $\mathcal{F}_+$ is torsion-free.
\end{proof}

\subsection{Proof for the case $n=2$
and 
$h^1(\mathcal{E})>0$}\label{Proof for the case n=2 and h^1(E)>0}
Set $s=h^1(\mathcal{E})$.
Then we have the following exact sequence 
\[0\to \mathcal{E}\to \mathcal{E}_0\to \mathcal{O}^{\oplus s}\to 0\]
which induces an isomorphism $H^0(\mathcal{O}^{\oplus s})\cong H^1(\mathcal{E})$.
It follows from \cite[Theorem 6.2.12 (ii)]{MR2095472} 
that $\mathcal{E}_0$ is a nef vector bundle of rank $r+s$ with first Chern class three,
second Chern class nine, and $h^1(\mathcal{E}_0)=0$.
Since $H^0(\mathcal{E})\cong H^0(\mathcal{E}_0)$, the image $\mathcal{E}_1$ of 
the evaluation map $H^0(\mathcal{E}_0)\otimes\mathcal{O}\to \mathcal{E}_0$ is contained in $\mathcal{E}$.
Denote by $\mathcal{F}$ the quotient of $\mathcal{E}$ by $\mathcal{E}_1$.
By abuse of notation, we also denote by $E_2^{-1,1}$ 
the cokernel of the evaluation map $H^0(\mathcal{E}_0)\otimes\mathcal{O}\to \mathcal{E}_0$.
Then we have the following commutative diagram with exact rows.
\[
\xymatrix{
0\ar[r]&\mathcal{E}      \ar[r]\ar[d]&\mathcal{E}_0\ar[r]\ar[d]&\mathcal{O}^{\oplus s}\ar[r]\ar@{=}[d]&0    \\
0\ar[r]&\mathcal{F}      \ar[r]      &E_2^{-1,1}\ar[r]         &\mathcal{O}^{\oplus s}\ar[r]          &0     
}
\]
Since we have a surjection $E_2^{-1,1}\to \mathcal{O}^{\oplus s}$ with $s\geq 1$,
Remark~\ref{cokernelOf16and17} 
and Lemmas~\ref{gamma1gabeta1woInshiNiMotu} and \ref{torsion-free}
imply that $E_2^{-1,1}$ is either $\mathcal{O}$ or a torsion-free sheaf of rank two with $c_1(E_2^{-1,1})=2$
and $c_2(E_2^{-1,1})=7$.
If $E_2^{-1.1}$ is the latter,
then $s=1$ and $\mathcal{F}\cong \mathcal{I}_Z(2)$ for some $0$-dimensional closed subscheme $Z$ of length seven.
Lemma~\ref{conicpassing5points} then shows that $\mathcal{E}$ admits a negative degree quotient, which contradicts 
that $\mathcal{E}$ is nef.
Hence $E_2^{-1,1}\cong \mathcal{O}$,
and thus $\mathcal{F}=0$, $\mathcal{E}_1=\mathcal{E}$,
and $\mathcal{E}_0$ and $\mathcal{E}_1$ fit in the following exact sequences
\begin{gather*}
0\to \mathcal{O}(-3)\to \mathcal{O}^{\oplus r+1}\to \mathcal{E}_1\to 0;\\
0\to \mathcal{E}_1\to \mathcal{E}_0\to \mathcal{O}\to 0
\end{gather*}
as in (\ref{E_300EssentiallyTangent-1}) and (\ref{EhasOasaQuotient}).
Hence $\mathcal{E}$ fits in an exact sequence
\[
0\to \mathcal{O}(-3)\to \mathcal{O}^{\oplus r+1}\to \mathcal{E}\to 0.
\]
This is the case (15) of Theorem~\ref{c_1=3c_2<8}, where $n=2$.

\subsection{Proof for the case $n\geq 3$}
What we have to show in case $n\geq 3$ is the following lemma.
\begin{lemma}\label{Oukyushochi}
If $n\geq 3$, then 
$c_3=27$ and 
$h^0(\mathcal{E}(-1))\leq 1$.
Moreover $\mathcal{E}$ satisfies one of the following:
\item[$(1)$] $n\geq 3$, $h^0(\mathcal{E}(-1))=0$, $h^{n-3}(\mathcal{E}(2-n))=0$, and $\mathcal{E}$ lies in the case $(15)$ of Theorem~\ref{c_1=3c_2<8};
\item[$(2)$] $n\geq 3$ and $h^0(\mathcal{E}(-1))=1$;
\item[$(3)$] $n\geq 4$, $h^0(\mathcal{E}(-1))=0$, and $h^{n-3}(\mathcal{E}(2-n))=1$.
\end{lemma}
\begin{proof}
We shall first show that $h^0(\mathcal{E}(-1))\leq 1$ by induction on $n\geq 3$.
It follows from (\ref{first assumption})
that $h^0(\mathcal{E}(-1))\leq h^0(\mathcal{E}|_H(-1))$ for any hyperplane $H$ in $\mathbb{P}^n$.
Note here that $h^0(\mathcal{E}|_{H}(-1))\leq 1$ by induction hypothesis if $n\geq 4$
and by what we have seen in \S\ref{Proof for the case n=2 and h^1(E)=0}
and \S\ref{Proof for the case n=2 and h^1(E)>0} if $n=3$.
Hence $h^0(\mathcal{E}(-1))\leq 1$.

Recall here that $\mathcal{E}$ is globally generated 
if $h^1(\mathcal{E}(-1))=0$
and $\mathcal{E}|_H$ is globally generated for any hyperplane $H$ in $\mathbb{P}^n$
by \cite[Lemma 3]{pswnef}.
In order to show that $\mathcal{E}$
lies in the case (15) of Theorem~\ref{c_1=3c_2<8},
it is enough to show that $h^0(\mathcal{E})=r+1$
and that $\mathcal{E}$ is globally generated.

Suppose that $n=3$. As we have seen in \S\ref{Proof for the case n=2 and h^1(E)>0},
we have $h^1(\mathcal{E}|_H)\leq 1$.

Suppose furthermore that $h^0(\mathcal{E}(-1))=1$. 
Then we get the case (2) of Lemma~\ref{Oukyushochi} in case $n=3$.
Note here that we have $c_3=27$ in this case. Indeed,
the argument above shows that $h^0(\mathcal{E}|_H(-1))=1$,
and this implies that $h^1(\mathcal{E}|_{H})=0$
as we have also seen in \S\ref{Proof for the case n=2 and h^1(E)=0}
and \S\ref{Proof for the case n=2 and h^1(E)>0}.
It then follows from (\ref{first vanishing}) and  (\ref{H^2vanishing}) 
that $h^q(\mathcal{E}(-1))=0$ for $q\geq 2$.
Now we have  
\[
1\geq 1-h^1(\mathcal{E}(-1))=
\chi(\mathcal{E}(-1))=(c_3-25)/2
\]
by (\ref{RRonP3(-1)}).
Note here that $c_3\geq 27$ by (\ref{selfintersection}).
Hence we infer that $c_3=27$ and that $h^1(\mathcal{E}(-1))=0$.

Suppose furthermore that $h^0(\mathcal{E}(-1))=0$.
It follows from $h^1(\mathcal{E}|_H)\leq 1$
and (\ref{first vanishing})  that $h^2(\mathcal{E}(-1))\leq 1$
and that equality holds if and only if $h^1(\mathcal{E}|_H)=1$.
Moreover we have $h^3(\mathcal{E}(-1))=0$
by (\ref{first vanishing}) and (\ref{H^2vanishing}).
It then follows from (\ref{RRonP3(-1)}) that 
\[
1\geq -h^1(\mathcal{E}(-1))+h^2(\mathcal{E}(-1))=
\chi(\mathcal{E}(-1))=(c_3-25)/2.
\]
Since $c_3\geq 27$ by (\ref{selfintersection}),
this implies that $c_3=27$, that $h^2(\mathcal{E}(-1))=1$, and that $h^1(\mathcal{E}(-1))=0$.
Hence $h^1(\mathcal{E}|_H)=1$.
As we have seen in \S\ref{Proof for the case n=2 and h^1(E)>0},
this implies that $\mathcal{E}|_H$ is globally generated
and that $h^0(\mathcal{E}|_H)=r+1$.
Therefore $\mathcal{E}$ is globally generated and $h^0(\mathcal{E})=r+1$.
This is the case (1) of Lemma~\ref{Oukyushochi} in case $n=3$.

Suppose that $n\geq 4$.
If $h^0(\mathcal{E}(-1))=1$, we get the case (2) of Lemma~\ref{Oukyushochi}.
Suppose that $h^0(\mathcal{E}(-1))=0$.
We shall show that $h^{n-3}(\mathcal{E}(2-n))\leq 1$ by induction on $n\geq 4$.
Note here that $h^q(\mathcal{E}(3-n))=0$ for all $q>0$
by (\ref{first vanishing}).
Hence $h^{n-3}(\mathcal{E}(2-n))\leq h^{n-4}(\mathcal{E}|_H(3-n))\leq 1$
by induction hypothesis if $n\geq 5$ and by what we have shown, i.e., $h^0(\mathcal{E}|_H(-1))\leq 1$ if $n=4$.
If $h^{n-3}(\mathcal{E}(2-n))= 1$, we obtain the case (3) of Lemma~\ref{Oukyushochi}.

Suppose that $n\geq 4$, that $h^0(\mathcal{E}(-1))=0$, and that $h^{n-3}(\mathcal{E}(2-n))= 0$.
We claim here that $\mathcal{E}$ lies in the case (15) of Theorem~\ref{c_1=3c_2<8}.
We proceed by induction 
not only on $n\geq 4$ but also on $n\geq 3$;
if $n=3$, then the two conditions $h^0(\mathcal{E}(-1))=0$ and $h^{n-3}(\mathcal{E}(2-n))= 0$
become the same and $\mathcal{E}$ lies in the case (15) of Theorem~\ref{c_1=3c_2<8} as we have seen above.
Suppose now that $n\geq 4$. The assumption $h^{n-3}(\mathcal{E}(2-n))= 0$ implies that 
$h^{n-4}(\mathcal{E}|_H(3-n))= 0$ by (\ref{first vanishing}) if $n\geq 5$
and by the assumption $h^0(\mathcal{E}(-1))=0$ if $n=4$.
Hence we infer that $h^0(\mathcal{E}|_{L^3}(-1))=h^{n-3}(\mathcal{E}(2-n))=0$ for any linear subspace $L^3$
of dimension three in $\mathbb{P}^n$.
Therefore we see that $h^0(\mathcal{E}|_H(-1))\leq h^0(\mathcal{E}|_{L^3}(-1))=0$
by (\ref{first assumption}).
Now it follows from the induction hypothesis that 
$\mathcal{E}|_H$ lies in the case (15) of Theorem~\ref{c_1=3c_2<8}.
Thus $h^0(\mathcal{E}|_H)=r+1$ and 
$\mathcal{E}|_H$ is globally generated.
Moreover $h^1(\mathcal{E}(-1))=0$ by (\ref{first vanishing})
and $h^0(\mathcal{E}(-1))=0$ by assumption.
Hence $h^0(\mathcal{E})=r+1$ and $\mathcal{E}$ is globally generated.
\end{proof}

\section{Several remarks on Theorem~\ref{c_1=3c_2<8}}\label{RmksOnMainTheorem}

\begin{rmk}\label{Rmk for the case (7)}
The exact sequence 
in the case $(7)$ of Theorem~\ref{c_1=3c_2<8}
induces the following
\[
0\to T_{\mathbb{P}^3}(-2)\to \mathcal{O}(1)\oplus\mathcal{O}^{\oplus r+2}\to \mathcal{E}\to 0,
\]
and, dualizing this, 
we obtain the following exact sequence
\[
0\to \mathcal{E}^{\vee}\to \mathcal{O}(-1)\oplus\mathcal{O}^{\oplus r+2}\to \Omega_{\mathbb{P}^3}(2)\to 0.
\]
Note that the injection $H^0(\mathcal{E}^{\vee})\to H^0(\mathcal{O}(-1)\oplus\mathcal{O}^{\oplus r+2})$
induces a splitting injection
$\mathcal{O}\otimes H^0(\mathcal{E}^{\vee})
\to \mathcal{O}^{\oplus r+2}$
and that the composite of two splitting injection
$\mathcal{O}\otimes H^0(\mathcal{E}^{\vee})
\to \mathcal{O}^{\oplus r+2}$ and 
$\mathcal{O}^{\oplus r+2}\to \mathcal{O}(-1)\oplus\mathcal{O}^{\oplus r+2}$
is equal to the composite of $\mathcal{O}\otimes H^0(\mathcal{E}^{\vee})\to \mathcal{E}^{\vee}$ and 
$\mathcal{E}^{\vee}\to \mathcal{O}(-1)\oplus\mathcal{O}^{\oplus r+2}$.
Thus $\mathcal{O}\otimes H^0(\mathcal{E}^{\vee})\to \mathcal{E}^{\vee}$ is also a splitting injection.
Hence $\mathcal{E}^{\vee}\cong \mathcal{E}_0^{\vee}\oplus \mathcal{O}\otimes H^0(\mathcal{E}^{\vee})$ 
for some vector bundle $\mathcal{E}_0$ of rank $s=r-h^0(\mathcal{E}^{\vee})$.
Since $c_3(\mathcal{E}_0)=2\neq 0$, we infer that $s\geq 3$.
Note that $h^0(\mathcal{E}_0^{\vee})=0$ and that $\mathcal{E}_0^{\vee}$ fits in an exact sequence
\[
0\to \mathcal{E}_0^{\vee}\to \mathcal{O}(-1)\oplus\mathcal{O}^{\oplus s+2}\to \Omega_{\mathbb{P}^3}(2)\to 0.
\]
Since $h^0(\Omega_{\mathbb{P}^3}(2))=6$ by the Bott formula \cite[p.\ 8]{oss},
we see that 
$s\leq 4$.
Moreover $h^1(\mathcal{E}_0^{\vee})=4-s$. 
The image of $\mathcal{E}_0^{\vee}\to \mathcal{O}(-1)\oplus\mathcal{O}^{\oplus s+2}
\to \mathcal{O}(-1)$ is $\mathcal{I}_Z(-1)$ for some closed subscheme $Z$ in $\mathbb{P}^3$.

Suppose that $Z=\emptyset$. Let $\mathcal{F}^{\vee}$ be the kernel of the surjection $\mathcal{E}_0^{\vee}\to \mathcal{O}(-1)$.
Then $\mathcal{F}$ fits in an exact sequence
\[
0\to T_{\mathbb{P}^3}(-2)\to \mathcal{O}^{\oplus s+2}\to \mathcal{F}\to 0,
\]
and $\mathcal{F}$ is a nef vector bundle of rank $s-1$ with $c_1(\mathcal{F})=2$. 
Moreover, as we have seen in \cite[Remark 6.7]{resolution},
$\mathcal{F}\cong \Omega_{\mathbb{P}^3}(2)$ if $s=4$ 
and $\mathcal{F}\cong \mathcal{N}(1)$ if $s=3$ where $\mathcal{N}$ is a null correlation bundle on $\mathbb{P}^3$.
Hence
$\mathcal{E}$ is either $\mathcal{O}(1)\oplus \Omega_{\mathbb{P}^3}(2)\oplus \mathcal{O}^{\oplus r-4}$
or $\mathcal{O}(1)\oplus \mathcal{N}(1)\oplus \mathcal{O}^{\oplus r-3}$
if $Z=\emptyset$.
\end{rmk}

\begin{rmk}\label{Rmk for the case (8)}
Suppose that $\mathcal{E}$ is in the case $(8)$ of Theorem~$\ref{c_1=3c_2<8}$.
Then $\mathcal{E}$ has $\mathcal{O}^{\oplus r-4}$ as a subbundle;
let $\mathcal{E}_0$ be the quotient bundle $\mathcal{E}/\mathcal{O}^{\oplus r-4}$ of rank four.
In \cite[\S 6 III (a)]{MR3119690}, it is stated that $\mathcal{E}_0\cong \Omega_{\mathbb{P}^4}(2)$.
$($Therefore we see that $\mathcal{E}\cong \Omega_{\mathbb{P}^4}(2)\oplus \mathcal{O}^{\oplus r-4}$.$)$

For the sake of completeness, we give a different proof of this result in our context.
First note that $\mathcal{E}_0$ fits in an exact sequence
\[
0
\to 
T_{\mathbb{P}^4}(-3)
\to
\mathcal{O}(-1)^{\oplus 10}\to \mathcal{O}^{\oplus 10}\to \mathcal{E}_0\to 0.
\]
Therefore we obtain the following exact sequence
\[
0
\to \mathcal{E}_0^{\vee}(-1)
\to \mathcal{O}(-1)^{\oplus 10}\to \mathcal{O}^{\oplus 10}
\to \Omega_{\mathbb{P}^4}(2)
\to 0.
\]
We split this sequence into the following two exact sequences:
\begin{gather}
0
\to \mathcal{E}_0^{\vee}(-1)
\to \mathcal{O}(-1)^{\oplus 10}
\to \mathcal{G}
\to 0;\label{DefExactSeqOfGasCokernelinOmage(2)preciseRmk}\\
0
\to \mathcal{G}
\to \mathcal{O}^{\oplus 10}
\to \Omega_{\mathbb{P}^4}(2)
\to 0.\label{DefExactSeqOfGasKernelinOmage(2)preciseRmk}
\end{gather}
We claim here that the induced map 
$H^0(\mathcal{O}^{\oplus 10})\to H^0(\Omega_{\mathbb{P}^4}(2))$ is an isomorphism.
Since $h^0(\Omega_{\mathbb{P}^4}(2))=10$ by the Bott formula \cite[p.\ 8]{oss},
it is enough to show that the map is injective.
Suppose, to the contrary, that there exists a subbundle $\mathcal{O}\to \mathcal{O}^{\oplus 10}$
such that the composite $\mathcal{O}\to \mathcal{O}^{\oplus 10}\to \Omega_{\mathbb{P}^4}(2)$ is zero.
Then the subbundle morphism $\mathcal{O}\to \mathcal{O}^{\oplus 10}$ induces a subbundle morphism
$\mathcal{O}\to \mathcal{G}$; let $\mathcal{G}_0$ be the quotient bundle $\mathcal{G}/\mathcal{O}$.
The composite of 
the subbundle morphism $\mathcal{O}\to \mathcal{G}$ and the extension class in $\Ext^1(\mathcal{G},\mathcal{E}_0^{\vee}(-1))$ 
corresponding to (\ref{DefExactSeqOfGasCokernelinOmage(2)preciseRmk})
lies in 
$H^1(\mathcal{E}_0^{\vee}(-1))$,
and it gives rise to  an exact sequence
\begin{equation}\label{ExactSeqinOmage(2)preciseRmk}
0\to \mathcal{E}_0^{\vee}\to \mathcal{F}\to \mathcal{O}(1)\to 0.
\end{equation}
Then $\mathcal{F}$ is a vector bundle, and 
it
also fits in an exact sequence
\[0\to \mathcal{F}\to \mathcal{O}^{\oplus 10}\to \mathcal{G}_0(1)\to 0.
\]
Hence $\mathcal{F}^{\vee}$ is nef.
On the other hand, it follows from $(\ref{ExactSeqinOmage(2)preciseRmk})$
that $c_3(\mathcal{F}^{\vee})=-2$ since $c_2=4$ and $c_3=2$.
This contradicts the non-negativity of the Chern classes of nef vector bundles.
Therefore the claim holds.
Hence we may assume that the dual of  $(\ref{DefExactSeqOfGasKernelinOmage(2)preciseRmk})$
is nothing but the exact sequence induced by the two wedge $\wedge^2 (\mathcal{O}^{\oplus 5})\to \wedge^2(T_{\mathbb{P}^4}(-1))$
of the surjection in the Euler exact sequence 
\begin{equation}\label{Euler}
0\to \mathcal{O}(-1)\to \mathcal{O}^{\oplus 5}\to T_{\mathbb{P}^4}(-1)\to 0.
\end{equation}
In particular, $\mathcal{G}\cong \Omega_{\mathbb{P}^4}^2(2)$.
Thus the exact sequence $(\ref{DefExactSeqOfGasCokernelinOmage(2)preciseRmk})$ implies an exact sequence
\begin{equation}\label{SaigoExactSeqinOmage(2)preciseRmk}
0
\to \mathcal{E}_0^{\vee}
\to \mathcal{O}^{\oplus 10}
\to \Omega_{\mathbb{P}^4}^2(3)
\to 0.
\end{equation}
Next we claim that the induced map 
$H^0(\mathcal{O}^{\oplus 10})\to H^0(\Omega_{\mathbb{P}^4}^2(3))$ is an isomorphism.
Since $h^0(\Omega_{\mathbb{P}^4}^2(3))=10$ by the Bott formula,
it is enough to show that the map is injective.
Suppose, to the contrary, that there exists a subbundle $\mathcal{O}\to \mathcal{O}^{\oplus 10}$
such that the composite $\mathcal{O}\to \mathcal{O}^{\oplus 10}\to \Omega_{\mathbb{P}^4}^2(3)$ is zero.
Then the subbundle morphism $\mathcal{O}\to \mathcal{O}^{\oplus 10}$ induces a subbundle morphism
$\mathcal{O}\to \mathcal{E}_0^{\vee}$; let $\mathcal{E}_1^{\vee}$ be the quotient bundle $\mathcal{E}_0^{\vee}/\mathcal{O}$.
Then $\mathcal{E}_1^{\vee}$ fits in an exact sequence
\[
0
\to \mathcal{E}_1^{\vee}
\to \mathcal{O}^{\oplus 9}
\to \Omega_{\mathbb{P}^4}^2(3)
\to 0.
\]
Hence it follows from the Bott formula that $h^0(\mathcal{E}_1)=9$.
Since $h^0(\mathcal{E}_0)=10$, this implies that $\mathcal{E}_0\cong \mathcal{E}_1\oplus \mathcal{O}$.
Thus $c_4(\mathcal{E}_0)=0$, which however contradicts 
that $c_4=1$.
Therefore $H^0(\mathcal{O}^{\oplus 10})\to H^0(\Omega_{\mathbb{P}^4}^2(3))$ is an isomorphism,
and we conclude that the exact sequence $(\ref{SaigoExactSeqinOmage(2)preciseRmk})$
is nothing but the exact sequence induced by the two wedge $\wedge^2 (\mathcal{O}^{\oplus 5})\to \wedge^2(T_{\mathbb{P}^4}(-1))$
of the surjection in the Euler exact sequence $(\ref{Euler})$.
Therefore
$\mathcal{E}_0\cong \Omega_{\mathbb{P}^4}(2)$.
\end{rmk}

\begin{rmk}\label{TrautmannVetter}
Suppose that $n=4$ and that 
$\mathcal{E}$ fits in an exact sequence in the case $(10)$ of Theorem~$\ref{c_1=3c_2<8}$.
Then $\mathcal{E}$ is an extension of the Tango bundle 
by a trivial bundle $\mathcal{O}^{\oplus r-3}$.

The reason is as follows. Since $\mathcal{E}$ is globally generated,
$\mathcal{E}$ has $\mathcal{O}^{\oplus r-4}$ as a subbundle; 
denote by $\mathcal{E}_0$ the quotient bundle $\mathcal{E}/\mathcal{O}^{\oplus r-4}$.
Since $\mathcal{E}_0$ is globally generated of rank four with 
$c_4(\mathcal{E}_0)=0$,
$\mathcal{E}_0$ has also $\mathcal{O}$ as a subbundle; 
denote by $\mathcal{E}_1$ the quotient bundle $\mathcal{E}_0/\mathcal{O}$.
We show that $\mathcal{E}_1$ is the Tango bundle.
First note that the dual $\mathcal{E}_1^{\vee}$ of  $\mathcal{E}_1$ fits in an exact sequence
\[0\to\mathcal{E}_1^{\vee} \to \mathcal{O}^{\oplus 7}\to\Omega_{\mathbb{P}^4}(2) \to 0.\]
Note also that $h^0(\mathcal{E}_1^{\vee})=0$;
indeed, if $h^0(\mathcal{E}_1^{\vee})\neq 0$, then 
$\mathcal{E}_1$ would admit $\mathcal{O}$ as a direct summand,
which contradicts the fact that $c_3(\mathcal{E}_1)=5\neq 0$
and the rank of $\mathcal{E}_1$ is three. 
Since $h^0(\Omega_{\mathbb{P}^4}(2))=10$ by the Bott formula~\cite[p.\ 8]{oss},
this implies that $h^1(\mathcal{E}_1^{\vee})=3$.
Now we have an isomorphism 
$\Ext^1(\Ext^1(\mathcal{O},\mathcal{E}_1^{\vee})\otimes \mathcal{O},\mathcal{E}_1^{\vee})
\cong 
\End(\Ext^1(\mathcal{O},\mathcal{E}_1^{\vee}))$;
let $\xi$ be the element in $\Ext^1(\Ext^1(\mathcal{O},\mathcal{E}_1^{\vee})\otimes \mathcal{O},\mathcal{E}_1^{\vee})$
corresponding to the identity
in $\End(\Ext^1(\mathcal{O},\mathcal{E}_1^{\vee}))$.
Consider the extension
\[0\to \mathcal{E}_1^{\vee}\to \mathcal{F}\to \Ext^1(\mathcal{O},\mathcal{E}_1^{\vee})\otimes \mathcal{O}\to 0\]
corresponding to $\xi$;
then $H^0(\mathcal{F})\cong H^0(\mathcal{E}_1^{\vee})=0$ and $H^1(\mathcal{F})=0$.
Let 
\[0\to \mathcal{O}^{\oplus 7}\to \mathcal{O}^{\oplus 10}\to \Ext^1(\mathcal{O},\mathcal{E}_1^{\vee})\otimes \mathcal{O}\to 0\]
be the extension corresponding to the image 
of $\xi$ via 
the map
\[\Ext^1(\Ext^1(\mathcal{O},\mathcal{E}_1^{\vee})\otimes \mathcal{O},\mathcal{E}_1^{\vee})
\to 
\Ext^1(\Ext^1(\mathcal{O},\mathcal{E}_1^{\vee})\otimes \mathcal{O},\mathcal{O}^{\oplus 7}).
\]
Then $\mathcal{F}$ fits in an exact sequence
\[
0\to \mathcal{F}\to \mathcal{O}^{\oplus 10}\to \Omega_{\mathbb{P}^4}(2)\to 0.
\]
Since $h^0(\mathcal{F})=h^1(\mathcal{F})=0$,
the induced map $H^0(\mathcal{O}^{\oplus 10})\to H^0(\Omega_{\mathbb{P}^4}(2))$ is an isomorphism.
Therefore $\mathcal{F}^{\vee}\cong \Omega_{\mathbb{P}^4}^2(3)$,
and thus $\mathcal{E}_1$ is the Tango bundle.

According to \cite[\S 4]{MR1987746},
Trautmann~\cite{MR0352523} and Vetter~\cite{MR0344518}
give 
an
explicit construction 
of 
the bundle which is,
up to taking duals and twists by $\mathcal{O}(1)$,
the Tango bundle.
\end{rmk}

In the following, we give an example in case (16) of Theorem~\ref{c_1=3c_2<8}.

\begin{ex}\label{Example of the case (16)}
Let $\psi:\mathcal{O}(-2)^{\oplus 3}\to \mathcal{O}(-1)^{\oplus 3}$ be the morphism defined by a matrix
\[
\begin{bmatrix}
0&y&x-\lambda z\\
x&0&y\\
z&z-x&0
\end{bmatrix},
\]
where $(x:y:z)$ are homogeneous coordinates of $\mathbb{P}^2$ and $\lambda\in K\setminus \{0,1\}$.
Then $\Coker(\psi)$ is supported on an elliptic curve $E:y^2z=x(x-z)(x-\lambda z)$,
and its Chern polynomial $c_t(\Coker(\psi))=1+3t+9t^2$.
Moreover $h^0(\Coker(\psi))=0$.
Take a sufficiently large integer $r$ (e.g., $r\geq 6$)
and a general morphism $\psi':\mathcal{O}(-2)^{\oplus 3}\to 
\mathcal{O}^{\oplus r}$,
and consider a subbundle morphism $\Psi={}^t[\psi',\psi]:\mathcal{O}(-2)^{\oplus 3}\to 
\mathcal{O}^{\oplus r}\oplus \mathcal{O}(-1)^{\oplus 3}$.
Let $\mathcal{E}$ be the cokernel of $\Psi$.
Then $\mathcal{E}$ fits in an exact sequence
\[0\to \mathcal{O}^{\oplus r}\xrightarrow{\varphi} \mathcal{E}\to \Coker(\psi)\to 0,\]
where 
$c_1(\mathcal{E})=3$ and $c_2(\mathcal{E})=9$.
The degeneracy locus $Z$ of the composite of a general inclusion $\mathcal{O}^{\oplus r-1}\hookrightarrow \mathcal{O}^{\oplus r}$
and $\varphi$ has codimension two, and we see that the cokernel of the composite is isomorphic to $\mathcal{I}_Z(3)$,
where $Z$ is a $0$-dimensional closed subscheme of length $9$
in $\mathbb{P}^2$. Moreover
we have the following exact sequence
\[0\to \mathcal{O}_{\mathbb{P}^2}\to \mathcal{I}_Z(3)\to \Coker(\psi)\to 0.\]
Therefore $Z$ lies on the elliptic curve $E$, and $\Coker(\psi)$ is isomorphic to $\mathcal{O}_E(\mathfrak{d})$,
where $\mathfrak{d}$ is a divisor of degree zero on $E$;
thus $\mathcal{E}$ fits in the following exact sequence
\[0\to \mathcal{O}^{\oplus r}\to \mathcal{E}\to \mathcal{O}_E(\mathfrak{d})\to 0.\]
This implies that $\mathcal{E}$ is nef.
Finally note that $\mathfrak{d}\neq 0$ since $h^0(\Coker(\psi))=0$.
\end{ex}

%%%%%%%%%%%%%%%%%%%%%%%%%%%%%%%%%%%%%%%%%%%%%%%%%%%%%%%%%%%%%
%%%Question 
%%%%%%%%%%%%%%%%%%%%%%%%%%%%%%%%%%%%%%%%%%%%%%%%%%%%%%%%%%%%%%%%%%%%
We end this section with the following question about some properties of 
a nef vector bundle $\mathcal{E}$ in 
case (16) of Theorem~\ref{c_1=3c_2<8}.
\begin{question}
Is the support of the evaluation map $H^0(\mathcal{E})\otimes \mathcal{O}\to \mathcal{E}$
in case (16) of Theorem~\ref{c_1=3c_2<8}
necessarily 
reduced, 
irreducible and nonsingular~?
\end{question}

\section{Nef but non-globally generated vector bundles}\label{nefbutNonGG}
\begin{lemma}\label{doubledualoftorsionfreequotientofnefvbonSmoothsurfaceisanefvb}
Let $\mathcal{F}$ be a nef vector bundle on a smooth projective surface $X$.
Let $\mathcal{E}_0$ be a torsion-free quotient of $\mathcal{F}$,
i.e., there exists a surjection $\mathcal{F}\to \mathcal{E}_0$ with $\mathcal{E}_0$ a torsion-free
coherent sheaf.
Let $\mathcal{E}$ denote the double dual $\mathcal{E}_0^{\vee\vee}$ of $\mathcal{E}_0$.
Then $\mathcal{E}$ is a nef vector bundle.
\end{lemma}
\begin{proof}
Since $\mathcal{E}$ is a reflexive sheaf on a smooth surface, 
$\mathcal{E}$ is a vector bundle.
To show that $\mathcal{E}$ is nef,
it is enough to show that, for any finite morphism $C\to X$
from a smooth curve $C$,
every quotient line bundle $\mathcal{L}$ of $\mathcal{E}|_C$
has non-negative degree.
Note here that the natural injection $\mathcal{E}_0\to \mathcal{E}$
induces a generically injective morphism $\mathcal{E}_0|_C\to \mathcal{E}|_C$.
Now let $\mathcal{M}$ be the image of the composite
of the morphism $\mathcal{E}_0|_C\to \mathcal{E}|_C$
and the surjection $\mathcal{E}|_C\to \mathcal{L}$.
Since the composite of $\mathcal{F}|_C\to \mathcal{E}_0|_C$ and $\mathcal{E}_0|_C\to \mathcal{M}$
is surjective
and $\mathcal{F}$ is nef, 
we see that $\mathcal{M}$ has non-negative degree.
Hence $\mathcal{L}$ has non-negative degree since
there is an injection $\mathcal{M}\to \mathcal{L}$ of line bundles on the smooth curve $C$.
Therefore $\mathcal{E}$ is nef.
\end{proof}

As is indicated by the statement in Lemma~\ref{key} (2) (a),
we can construct a nef but non-globally generated vector bundle on $\mathbb{P}^2$
if $c_1=3$ and $c_2=8$.
See also Example~\ref{ExampleOfNonGGWithc_2=9} in \S\ref{PropertyOfNefNonGG}
(besides Example~\ref{Example of the case (16)} in the previous section)
for an example of 
a nef but non-globally generated vector bundle on $\mathbb{P}^2$
with $c_1=3$ and $c_2=9$.

%\begin{theopargself}
%\begin{proof}[of Proposition~\ref{exampleofnefbutnonGGvb}]
%%%Use the following form if you do not use springer theorem package
\begin{proof}[Proof of Proposition~\ref{exampleofnefbutnonGGvb}]
Given an integer $r\geq 2$ and a closed point $w$ in $\mathbb{P}^2$, 
note first that
there exists a section $s$ in $H^0(\mathcal{O}(3)^{\oplus r+1})$
such that the zero locus $(s)_0$ of $s$ is $\{w\}$ as closed subschemes.
Let $\varphi:\mathcal{O}(-3)\to \mathcal{O}^{\oplus r+1}$ be the morphism
determined by $s$,
and 
$\mathcal{E}_0$ 
the cokernel of $\varphi$.
The dual $\varphi^{\vee}:\mathcal{O}^{\oplus r+1}\to \mathcal{O}(3)$
of $\varphi$ has $\mathcal{I}_w(3)$
as its image,
and we obtain an exact sequence
\[
0\to \mathcal{E}_0^{\vee}\to \mathcal{O}^{\oplus r+1}\to \mathcal{I}_w(3)\to 0.
\]
On the other hand, the ideal sheaf $\mathcal{I}_w$ sits in an exact sequence
\begin{equation}\label{resolutionOfmaximalidealofp}
0\to \mathcal{O}(-2)\to \mathcal{O}(-1)^{\oplus 2}\to \mathcal{I}_w\to 0.
\end{equation}
Therefore
$\Tor_i^{\mathcal{O}_w}(\mathcal{I}_w(3),k(w))=0$ for $i>1$.
Hence 
$\mathcal{E}_0^{\vee}$ is a vector bundle.
The exact sequence (\ref{resolutionOfmaximalidealofp})
also implies that
$\calext^1(\mathcal{I}_w(3),\mathcal{O})\cong k(w)$.
Let $\mathcal{E}$ be the double dual $\mathcal{E}_0^{\vee\vee}$ of $\mathcal{E}_0$.
Since $\calhom(\mathcal{I}_w(3),\mathcal{O})\cong \mathcal{O}(-3)$,
the vector bundle $\mathcal{E}$ fits in 
the desired 
exact sequence
\begin{equation}\label{PropRequiredExSeq}
0\to \mathcal{O}(-3)\xrightarrow{\varphi} \mathcal{O}^{\oplus r+1}\to \mathcal{E}
\to 
k(w)
\to 0.
\end{equation}

Suppose next that a vector bundle $\mathcal{E}$ fits in  the exact sequence (\ref{PropRequiredExSeq}).
We split the sequence (\ref{PropRequiredExSeq}) into the following two exact sequences:
\begin{gather*}
0\to \mathcal{O}(-3)\to \mathcal{O}^{\oplus r+1}\to \mathcal{E}_0\to 0;\\
0\to \mathcal{E}_0\to \mathcal{E}\to 
k(w)
\to 0.
\end{gather*}
We see that $\mathcal{E}_0$ is a torsion-free sheaf of rank $r$
with $c_1(\mathcal{E}_0)=3$, $c_2(\mathcal{E}_0)=9$, 
and $h^1(\mathcal{E}_0)=1$.
We have
$\mathcal{E}\cong \mathcal{E}_0^{\vee\vee}$,
and thus $\mathcal{E}$ is a nef vector bundle by 
Lemma~\ref{doubledualoftorsionfreequotientofnefvbonSmoothsurfaceisanefvb}.
Moreover $c_1=3$ and $c_2=8$. 
Since $c_2<9$, we obtain $h^1(\mathcal{E})=0$ by (\ref{KVvanishing}).
Hence $H^0(\mathcal{E}_0)\cong H^0(\mathcal{E})$. Therefore $\mathcal{E}$ is not globally generated.
\end{proof}
%\end{theopargself}

\begin{rmk}
If $r=2$, the exact sequence in Proposition~\ref{exampleofnefbutnonGGvb} already appears in \cite[3.2.5]{MR1633159}.
Professor Adrian Langer kindly informed the author of this fact and that he ruled out this case by mistake.
\end{rmk}

\section{
Some examples
}\label{PropertyOfNefNonGG}
Let $X$ be a smooth projective variety, and $\mathcal{E}$ a vector bundle on $X$ of rank $r$.
It is well known (see, e.g., \cite[Statement(folklore) 4.1]{MR1172165})
that if $\mathcal{E}$ is globally generated
then $r-1$ general global sections of $\mathcal{E}$ define an injection
$\mathcal{O}_X^{\oplus r-1}\to \mathcal{E}$ and this injection extends to an exact sequence
\[0\to \mathcal{O}_X^{\oplus r-1}\to \mathcal{E}\to \mathcal{I}_Z\otimes \det\mathcal{E}\to 0
\]
where $Z$ is 
a locally complete intersection subscheme 
of codimension two in $X$, if not empty.
For nef vector bundles, however, analogous results do not hold in general,
even if $h^0(\mathcal{E})\geq r-1$, as the following examples show.

\begin{ex}
Let $\mathcal{E}_0$ be a nef vector bundle of rank $r-2$ on $\mathbb{P}^2$ fitting in the following exact sequence
\[0\to \mathcal{O}(-4)\to \mathcal{O}^{\oplus r-1}\to \mathcal{E}_0\to 0.\]
Then $r\geq 4$ and $h^1(\mathcal{E}_0)=3$.
Let $\xi_1$ and $\xi_2$ be linearly independent elements in $H^1(\mathcal{E}_0)$,
and let 
\[0\to \mathcal{E}_0\to \mathcal{E}\to \mathcal{O}^{\oplus 2}\to 0\]
be the exact sequence whose extension class in 
$\Ext^1(\mathcal{O}^{\oplus 2},\mathcal{E}_0)$ is determined by $\xi_1$ and $\xi_2$.
Then the connecting homomorphism $H^0(\mathcal{O}^{\oplus 2})\to H^1(\mathcal{E}_0)$ is injective,
and thus $h^0(\mathcal{E})=h^0(\mathcal{E}_0)=r-1$.
Moreover $\mathcal{E}$ is a nef vector bundle of rank $r$ by \cite[Theorem 6.2.12 (ii)]{MR2095472}.
In this example,
every morphism $\mathcal{O}^{\oplus r-1}\to \mathcal{E}$ is not injective,
since it factors through the bundle $\mathcal{E}_0$ of rank $r-2$.
\end{ex}

\begin{ex}\label{ExampleOfNonGGWithc_2=9}
Let $\mathcal{E}_0$ be a nef vector bundle of rank $r-1$ on $\mathbb{P}^2$ fitting in the following exact sequence
\[0\to \mathcal{O}(-3)\to \mathcal{O}^{\oplus r}\to \mathcal{E}_0\to 0.\]
Then $r\geq 3$ and $h^1(\mathcal{E}_0)=1$.
Let 
\[0\to \mathcal{E}_0\to \mathcal{E}\to \mathcal{O}\to 0\]
be a non-split exact sequence;
the connecting homomorphism $H^0(\mathcal{O})\to H^1(\mathcal{E}_0)$ is an isomorphism.
Then $h^0(\mathcal{E})=h^0(\mathcal{E}_0)=r$,
and it follows from \cite[Theorem 6.2.12 (ii)]{MR2095472} that 
$\mathcal{E}$ is a nef but non-globally generated vector bundle of rank $r$ with $c_1=3$ and $c_2=9$.
In this example, a general morphism $\mathcal{O}^{\oplus r-1}\to \mathcal{E}$ is injective,
but its cokernel $\mathcal{C}$ fits in a non-split exact sequence
\[
0\to \mathcal{O}_C\to \mathcal{C}\to \mathcal{O}\to 0
\]
where $\mathcal{O}_C$ is the structure sheaf of some curve $C$ of degree $3$ in $\mathbb{P}^2$;
since $\mathcal{C}$
has a non-zero torsion subsheaf $\mathcal{O}_C$, $\mathcal{C}$ is not isomorphic to a torsion-free coherent sheaf 
$\mathcal{I}_Z\otimes \det\mathcal{E}$ for any closed subscheme $Z$ of $\mathbb{P}^2$.
\end{ex}

%\renewcommand{\thefootnote}{\fnsymbol{footnote}}
%\footnote[0]{2010 {\it Mathematics Subject Classification}\/: Primary 
%14F05;
%Secondary 
%14J60}
%\footnote[0]{{\it Key words and Phrases}\/: nef vector bundles,
%spectral sequences, Fano bundles}
%%\footnote[0]{
%%}

\bibliographystyle{acm}
%\bibliographystyle{plain}
%\bibliographystyle{alpha}
%\bibliography{myrefsver25}
\bibliography{Rev7EprintNewNefOnProjSpace.bbl}
\end{document}